\documentclass[12pt, a4paper]{article}

\usepackage{a4,amsmath,amssymb, amsthm, latexsym, color, graphicx,url}
\usepackage{longtable}
\usepackage{tikz}
\usetikzlibrary{arrows.meta}
\usetikzlibrary{decorations.markings}
\usetikzlibrary{calc}

\usepackage{fullpage}
\usepackage{enumitem}
\usepackage{xcolor}
\newtheorem{theorem}{Theorem}[section]
\newtheorem{proposition}[theorem]{Proposition}
\newtheorem{lemma}[theorem]{Lemma}
\newtheorem{corollary}[theorem]{Corollary}

\theoremstyle{definition}
\newtheorem{example}[theorem]{Example}
\newtheorem{definition}[theorem]{Definition}
\newtheorem{remark}[theorem]{Remark}

\title{Internal and external partial difference families and cyclotomy}
\author{Sophie Huczynska and Laura Johnson \thanks{email: sh70@st-andrews.ac.uk,  lj68@st-andrews.ac.uk}}
\date{School of Mathematics \& Statistics, University of St Andrews, St Andrews, KY16 9SS, Scotland, UK }
\begin{document}
\maketitle

\begin{abstract}
We introduce the concept of a disjoint partial difference family (DPDF) and an external partial difference family (EPDF), a natural generalization of the much-studied disjoint difference family (DDF), external difference family (EDF) and partial difference set (PDS).  We establish properties and indicate connections to other recently-studied combinatorial structures.  We show how DPDFs and EPDFs may be formed from PDSs, and present various cyclotomic constructions for DPDFs and EPDFs.  As part of this, we develop a unified cyclotomic framework, which yields some known results on PDSs, DDFs and EDFs as special cases.
\end{abstract}

\section{Introduction}

Difference sets and difference families (sets, or families of sets, in which every non-zero group element occurs with constant frequency as an internal difference within the sets) were introduced in the 1930s and have been very well-explored in the literature.  Difference families are useful for constructing balanced incomplete block designs via development (see \cite{ColDin}, \cite{Wil}).  Disjoint difference families (difference families whose sets are disjoint) have recently received particular attention \cite{Bur}: there is design theory motivation for asking whether it is always possible to find disjoint base blocks for designs  (see Novak's conjecture and its generalization in \cite{Nov}, \cite{FeHoWa}), and there are information theory applications (\cite{NgPa}).   In the early 2000s, external difference families were introduced (\cite{OgaKurStiSai}, \cite{PatSti}), in which each non-zero element occurs with constant frequency as an external difference between distinct sets.  These structures have strong links to cryptography.  Finally, the concept of a partial difference set  (a set in which each non-zero element occurs with one of two frequencies according to whether it lies in the set) is a natural generalization of a difference set, which arises naturally and has various useful applications (\cite{Ma84}, \cite{Ma})).

However the partial analogue of a difference family has not previously been studied.  Related ideas are present in the literature; for example, work has been done on structures called almost difference families (\cite{DinYin}), where the non-identity elements occur with two frequencies $\lambda$ and $\lambda+1$; here the sets need not be disjoint and there is no structural condition determining the sets of elements which occur with each frequency.   A specialized version of an external partial difference family was introduced in \cite{DavHucMul}, in which the sets partition the non-zero group elements and the two frequencies are associated with membership/non-membership of unions of the sets in the family.  Finally, ideas of the type explored in this paper occur implictly in \cite{MutTon}, where the authors seek bounded versions of external difference families in $\mathbb{Z}_n$ (called difference systems of sets) and use conditions on internal differences to guarantee lower-bounded external differences.  We set up a more general version of both disjoint difference families and external difference families, precisely analogous to the definition of partial difference set. We will see these objects arise naturally.

Much of our approach will be using cyclotomy in finite fields.  There is a long history of using cyclotomic methods to construct internal and external difference families, dating back to \cite{Wil}; see for example  \cite{Bao}, \cite{ChaDin}, \cite{CheLinLin}, \cite{DavHucMul}, \cite{HuaWu}, \cite{Li}, \cite{MutTon}, \cite{WenYanFuFen} and \cite{Xia}.  Difference families whose sets are cyclotomic classes are known as radical (\cite{Bur2}).  

Throughout, we let $G$ be a group, written additively, and let $G^*$ denote $G \setminus \{0\}$.  For a subset $D$ of $G$, we define the multiset $\Delta(D)=\{ x-y: x \neq y \in D \}$ and for disjoint sets $D_1, D_2 \subseteq G$, we define the multiset $\Delta(D_1, D_2)=\{ x-y: x \in D_1, y \in D_2 \}$.  For a subset $D$ of $G$ and non-negative integer $\lambda$, we denote by $\lambda D$ the multiset comprising $\lambda$ copies of each element of $D$.

The following definitions are well-known (see \cite{PatSti}, \cite{Ma84}); we follow the notation of \cite{PatSti}.
\begin{definition}\label{def_basic}
Let $G$ be a group of order $n$.
\begin{itemize}
\item[(i)] A subset $D$ of $G$, with cardinality $k$, forms an $(n,k,\lambda)$-Difference Set (DS) if the following multiset equation holds:
$$ \Delta(D) = \lambda(G^*).$$

\item[(ii)] A subset $P$ of $G$, with cardinality $k$, forms an $(n,k,\lambda,\mu)$-Partial Difference Set (PDS) if the following multiset equation holds:
$$ \Delta(P) = \lambda(P) + \mu(G^*\backslash{P}).$$
Note this is a Difference Set if $\lambda=\mu$.  We shall say that a PDS is \emph{proper} if $\lambda \neq \mu$. 

\item[(iii)] A collection of $m$ disjoint $k$-subsets $\mathcal{S}^{\prime} = \{D_1,...,D_m\}$ in $G$ forms an $(n,m,k,\lambda)$-Disjoint Difference Family (DDF) of $G$ if the following multiset equation holds;
$$\cup_{i=1}^m \Delta(D_i) = \lambda(G^*).$$
If we do not require the subsets to be pairwise disjoint, this structure is called a Difference Family (DF).  

\item[(iv)] A collection of $m$ disjoint $k$-subsets $\mathcal{S}^{\prime} = \{D_1,\ldots,D_m\}$ in \rm{G} forms an $(n,m,k,\lambda)$-External Difference Family (EDF) of \rm{G} if the following multiset equations holds;
\begin{equation*}
\cup_{i,j: i \neq j} \Delta(D_i,D_j)=\lambda(G^*).
\end{equation*}

\item[(v)] A collection of $m$ disjoint $k$ subsets $\mathcal{S}^{\prime} = \{D_1,\ldots,D_m\}$ in \rm{G} forms an, $(n,m,k,\lambda)$-Strong External Difference Family (SEDF) of \rm{G} if the following multiset equation holds for each $D_i \in \mathcal{S}^{\prime}$;
\begin{equation*}
\cup_{j: j \neq i} \Delta(D_i,D_j) = \lambda(\rm{G}^*).
\end{equation*}
If the set-sizes are allowed to vary, with $|A_i|=k_i$, and $\lambda$ is replaced by $\lambda_i$ in the above equation, then the structure is called a Generalised Strong External Difference Family (GSEDF).

\item[(vi)] A collection of $m$ disjoint subsets $\mathcal{S}^{\prime} = \{D_1,\ldots,D_m\}$ in \rm{G}, comprising $c_h$ subsets of size $k_h$ ($1 \leq h \leq \ell$) with $m=\sum_i c_i$, forms an $(n,m; c_1,\ldots, c_{\ell}; k_1, \ldots, k_{\ell}; \lambda_1, \ldots, \lambda_{\ell})$-Partitioned External Difference Family (PEDF) of \rm{G} if the following multiset equations holds for each $1 \leq h \leq \ell$:
\begin{equation*}
\cup_{i: |A_i|=k_h} \cup_{j: j \neq i} \Delta(D_i,D_j)=\lambda_i (G^*).
\end{equation*}
Both EDFs and GSEDFs are special cases of PEDFs.
\end{itemize}
\end{definition}

We make the following new definitions.

\begin{definition}\label{PDF}
Let $\mathcal{S}^{\prime}$ be a collection of $m$ disjoint subsets $\{ D_1, \ldots, D_m \}$ in $G^*$, where each $D_i$ has cardinality $k$, and let $S=\cup_{i=1}^m D_i$.  Then $\mathcal{S}^{\prime}$ is said to form an $(n,m, k, \lambda, \mu)$ Disjoint Partial Difference Family (DPDF) of $G$ if the following multiset equation holds: 
$$ \cup_{i=1}^m \Delta(D_i)= \lambda S + \mu (G^* \setminus S). $$
If $\lambda \neq \mu$ then $\mathcal{S}^{\prime}$ is called \textit{proper}. 
\end{definition}

\begin{definition}\label{EPDF}
Let $\mathcal{S}^{\prime}$ be a collection of $m$ disjoint subsets $\{ D_1, \ldots, D_m \}$ in $G^*$, where each $D_i$ has cardinality $k$, and let $S=\cup_{i=1}^m D_i$.  Then $\mathcal{S}^{\prime}$ is said to form an $(n,m, k, \lambda, \mu)$ External Partial Difference Family (EPDF) of $G$ if the following multiset equation holds: 
$$ \cup_{i,j: i \neq j} \Delta(D_i, D_j)= \lambda S + \mu (G^* \setminus S).$$
If $\lambda \neq \mu$ then $\mathcal{S}^{\prime}$ is called \textit{proper}.
\end{definition}

In general we consider $S \neq G^*$ but we adopt the convention that taking $S=G^*$ corresponds to the situation when $\mathcal{S}^{\prime}$ is a DDF or EDF.

If $\lambda=\mu$, we obtain a DDF and EDF respectively in the above definitions. If $m=1$ we obtain a PDS in Definition \ref{PDF}. Some special cases of DPDFs give examples of almost difference families (see \cite{DinYin}), while EPDFs give examples of bounded external difference families (see \cite{PatSti}).  Finally, these objects (and the cyclotomic construction approach) are also related to supplementary difference sets \cite{Wal}.  We note that the internal and external partial difference families defined in this paper are not related to the concept of a difference family being partial due to having a short base block.

\begin{example}\label{ex:13}
Let $G=(GF(13),+)$.  Consider the subsets $D_1=\{1,3,9\}$ and $D_2=\{4,10,12\}$ and let $\mathcal{S}^{\prime}=\{D_1,D_2\}$, so $S=\{1,3,4,9,10,12\}$.  Then the multiset $\Delta(D_1) \cup\Delta(D_2)$ consists of two copies of $\{2,5,6,7,8,11\}$ (the complement of $S$ in $GF(13)^*$).  The multiset of external differences $\Delta(D_1,D_2) \cup \Delta (D_2,D_1)$ comprises two copies of $S$ and one copy of its complement.  Hence  $\mathcal{S}^{\prime}$ is a $(13,2,3,0,2)$-DPDF and a $(13,2,3,2,1)$-EPDF.
\end{example}

In this paper, we establish properties and exhibit constructions for DPDFs and EPDFs and show connections to other recently-studied combinatorial structures.  Section 2 presents basic results.  From Section 3 onwards, we focus on DPDFs and EPDFs in finite fields $GF(q)$.  We develop a cyclotomic framework which allows us to obtain structural information and thereby derive results beyond those obtainable solely from specific values of cyclotomic numbers.  Let $C_i^e$ denote the cyclotomic class of order $e$; we focus on two main questions:
\begin{itemize}
\item[(i)] Is (a subset of) $\{C_0^e, C_1^e, \ldots C_{e-1}^e\}$ a DPDF or EPDF if $C_0^e$ is a PDS? 
\item[(ii)] Is $\{C_0^e, C_{\epsilon}^e, \ldots C_{e-\epsilon}^e\}$ a DPDF or EPDF if $C_0^{\epsilon}=\cup_{j=0}^{e/\epsilon - 1} C_{j \epsilon}^e$ is a PDS (including DS)?
\end{itemize}
We also investigate conditions under which a collection of disjoint $\cup_{i \in I} C_i^e$ forms a DPDF or EPDF.  We prioritise structures which are simultaneously DPDFs and EPDFs; where possible, we aim to establish DPDF/EPDF results whose parameters may be obtained without direct field element calculation.   In so doing, we prove a generalization of a result on PDSs by Calderbank and Kantor \cite{CalKan}.  We also obtain various DPDF/EPDF results for specific values of $e$ using values of cyclotomic numbers.   
 
We will indicate in the text when known results occur as special cases of our approach.   While none of our proofs depend on computation, the computer algebra system GAP \cite{GAP} was used to find and verify examples.
 
\section{Basic results}

We begin by establishing some basic definitions and results.

\begin{definition}\label{def1}
Let $\mathcal{S}^{\prime}=\{D_1,D_2,...,D_m\}$ denote a collection of $m$ pairwise disjoint subsets of $G$. We define:
\begin{itemize}
\item[(i)] ${\rm Int}(\mathcal{S}^{\prime}) = \bigcup\limits^{m}_{i=1}\Delta(D_i)$
\item[(ii)]
${\rm Ext}(\mathcal{S}^{\prime}) = \bigcup\limits^{m}_{i,j:i\neq{j}}\Delta(D_i,D_j)$.
\end{itemize}
\end{definition}

\begin{remark}\label{rem1}
Let $S$ be a subset of a group $G$ partitioned by $\mathcal{S}^{\prime}=\{D_1,...,D_m\}$, where $\mathcal{S}^{\prime}$ is a collection of $m$ disjoint $k$-element subsets.  Then the following multiset equation holds:
$${\rm Int}(\mathcal{S}^{\prime}) \cup {\rm Ext}(\mathcal{S}^{\prime})= \Delta(S).
$$
\end{remark}

It is proved in \cite{ChaDin} that if $\mathcal{S}^{\prime}$ partitions $G$ (respectively, $G^*$) then it is an EDF if and only if it is a DDF.  In fact this holds when $\mathcal{S}^{\prime}$ partitions any difference set.

\begin{theorem}\label{thm1}
If $\mathcal{S}^{\prime}$ partitions an $(n,mk,\lambda)$-difference set $D$ then it is an $(n,m,k,\mu)$-DDF if and only if it is an $(n,m,k,\lambda-\mu)$-EDF.
\end{theorem}

We now ask: what is the situation when an EDF or DDF partitions a PDS?  We see that the concepts of DPDF and EPDF arise naturally here:

\begin{theorem}\label{thm2}
\begin{itemize}
\item[(i)]Let $P$ be an $(n,mk,\sigma,\mu)$-Partial Difference Set partitioned by the sets $\mathcal{S}^{\prime}=\{D_1, \ldots, D_m \}$. Then $\mathcal{S}^{\prime}$ is an $(n,m,k, \lambda)$-EDF if and only if $\mathcal{S}^{\prime}$ is a proper $(n,m,k,\sigma-\lambda, \mu-\lambda)$-DPDF.
\item[(ii)] Let $P$ be an $(n,mk,\sigma,\mu)$-Partial Difference Set partitioned by the sets $\mathcal{S}^{\prime}=\{D_1,\ldots,D_m\}$. Then $\mathcal{S}^{\prime}$ is an $(n,m,k,\lambda)$-DDF if and only if $\mathcal{S}^{\prime}$ is a proper $(n,m,k,\sigma-\lambda,\mu-\lambda)$-EPDF.
\end{itemize}
\end{theorem}

\begin{proof}
\begin{itemize}
\item[(i)] Since $P$ is an $(n,mk,\sigma,\mu)$-Partial Difference Set partitioned by $\mathcal{S}^{\prime} = \{D_1, \ldots, D_m\}$, by Remark \ref{rem1}, this means:

\begin{equation}
{\rm{Int}}(\mathcal{S}^{\prime}) \cup{\rm Ext}(\mathcal{S}^{\prime}) = \sigma(P) + \mu(G^*\backslash{P}).
\end{equation}

If $\mathcal{S}^{\prime}$ forms an $(n,m,k, \lambda)$-EDF, this then implies $\lambda \leq min(\sigma, \mu)$ and:

\begin{equation}
{\rm{Ext}}(\mathcal{S}^{\prime}) = \bigcup\limits^{m}_{i,j:i\neq{j}}\Delta(D_i,D_j) = \lambda(G^*). 
\end{equation}
Hence
\begin{equation}
{\rm{Int}}(\mathcal{S}^{\prime}) = (\sigma-\lambda)(P) + (\mu-\lambda)(G^*\backslash{P}).   
\end{equation}
Thus $\mathcal{S}^{\prime}$ is a proper $(n,m,k,\sigma-\lambda,\mu-\lambda)$-DPDF. The reverse direction is similar.

\item[(ii)] Similar to part (i).
\end{itemize}
\end{proof}

A similar result holds in the more general situation for DPDFs and EDPFs whose union form a difference set or partial difference set.  
\begin{theorem}\label{thm3}
\begin{itemize}
\item[(i)] Let $D$ be an $(n,mk,\lambda)$-DS and let $S'=\{D_1,...,D_m\}$ be a partition of $D$. Then $S'$ is an $(n,m,k,\sigma,\mu)$-EPDF if and only if $S'$ is an $(n,m,k,\lambda-\sigma,\lambda-\mu)$-DPDF.
\item[(ii)] Let $P$ be an $(n,mk,\lambda,\mu)$-PDS and let $S'=\{D_1,...,D_m\}$ be a partition of $P$.  Then $S'$ is an $(n,m,k,\lambda^*,\mu^*)$-EPDF if and only if $S'$ is an $(n,m,k,\lambda-\lambda^*,\mu-\mu^*)$-DPDF.
\end{itemize}
\end{theorem}

An illustration of DPDF and an EPDF partitioning a partial difference set is shown in Example \ref{ex:13}, while examples partitioning a difference set are given in Remark \ref{Hadamard}.

Finally, we observe that an object which is simultaneously a DPDF and an EPDF has the following structural requirement:

\begin{theorem}\label{thm4}
If $S'=\{D_1,...,D_m\}$ is a partition of a set $S$, such that $S'$ is an $(n,m,k,\lambda^*,\mu^*)$-DPDF and an $(n,m,k,\lambda,\mu)$-EPDF, then $S$ is an $(n,mk,\lambda+\lambda^*,\mu+\mu^*)$-PDS when $\lambda+\lambda^* \neq \mu+\mu^*$ and an $(n,mk,\sigma)$-DS when $\lambda+\lambda^* = \mu+\mu^* = \sigma$.
\end{theorem}
This motivates our approach to finding structures which are simultaneously DPDFs and EPDFs.

Clearly partial difference sets are of crucial importance when constructing and understanding DPDFs. PDSs have received attention in the literature and various results are known - see \cite{Ma84} and \cite{Ma}. We give a summary of basic results from \cite{Ma84}.  
\begin{theorem}\label{thm5}
Let $G$ be a group of order $n$ and let $D$ be a subset of $G$.
\begin{itemize}
\item[(i)] If D is a difference set then it is a partial difference set; its complement $G \setminus D$ is also a difference set and hence a partial difference set.
\item[(ii)] If D is a partial difference set which is not a difference set then $D=-D$.
\item[(iii)] If $D$ is a partial difference set with $D=-D$, then $D\setminus \{0\}, D \cup \{0\}, G \setminus D, (G \setminus D) \setminus \{0\}$ and $(G \setminus D) \cup \{0\}$ are partial difference sets.
\item[(iv)] If $D$ is a non-trivial subgroup of $G$, of order $k$, then it is a $(n,k,k,0)$-PDS.
\item[(v)] If $D$ is an $(n,k,\lambda,0)$-PDS with $\lambda \neq 0$ then $D \cup \{0\}$ is a subgroup of $G$.
\end{itemize}
\end{theorem}
A partial difference set $D$ in a group $G$ is traditionally called \emph{regular} if $D=-D$ and $0 \not\in D$. 

We end this section with another situation where DPDFs naturally arise.  In \cite{PatSti}, it was proven that
\begin{theorem}
Suppose $A_1,\ldots,A_m$ is a partition of $G$ (where $|G| = n$) such that there are $c_h$ subsets of size $k_h$ for $1 \leq h \leq \ell$. Then $A_1,\ldots,A_m$ is an $(n,m;c_1,\ldots,c_{\ell};k_1,\ldots,k_{\ell};\lambda_1,\ldots,\lambda_{\ell})$-PEDF if and only if the subsets of cardinality $k_h$ form an $(n,c_h,k_h,c_hk_h-\lambda_h)$-DDF in G for $1 \leq h \leq \ell$.
\end{theorem}
We have the following new result:

\begin{theorem}
Suppose $A_1,\ldots,A_m$ is a partition of ${G}^*$ (where $|{G}|=n$) such that there are $c_h$ subsets of size $k_h$ for $1 \leq h \leq \ell$. Then $A_1,\ldots,A_m$ is an $(n,m;c_1,\ldots,c_{\ell};k_1,\ldots,k_{\ell};\lambda_1,\ldots,\lambda_{\ell})$-PEDF if and only if the subsets of cardinality $k_h$ form an $(n,c_h,k_h,c_hk_h-\lambda_h-1,c_hk_h-\lambda_h)$-DPDF in G for $1 \leq h \leq \ell$. 
\end{theorem}

These two theorems, taken together, extend known results on GSEDFs which partition $G$ (\cite{PatSti}) and $G^*$ (\cite{WenYanFuFen}): these show that if a family of sets $\{D_1, \ldots, D_m \}$ partitions $G$ (respectively $G^*$), then it is a GSEDF if and only if each $D_i$ is an $(n,k_i,k_i-\lambda_i)$-DS (respectively $(n,k_i,\lambda_i,\mu_i)$-PDS with $\lambda_i=\mu_i-1$).

\subsection{Constructing DPDFs and EPDFs as collections of PDSs}

Various types of EDF-like object are obtainable by taking the sets in the family to be appropriate difference sets or partial difference sets.  The approach of taking a suitable collection of partial difference sets may be used to construct DPDFs, and if their union has appropriate properties, then an EPDF is also obtained.

\begin{theorem}\label{thm6a}
Let $G$ be a group of order $n$ and let $\mathcal{S}^{\prime}=\{D_1, \ldots, D_m\}$ be a collection of disjoint $k$-subsets of $G$ such that each $D_i$ is an $(n,k,\lambda, \mu)$-PDS.  Denote $S= \cup_{i=1}^m D_i$. Then
\begin{itemize}
    \item[(i)] $\mathcal{S}^{\prime}$ is an $(n,m,\lambda+(m-1)\mu,m\mu)$-DPDF;
    \item[(ii)] If $S$ forms an  $(n,mk,\sigma,\chi)$-PDS (or $(n,mk,\sigma)$-DS) then $\mathcal{S}^{\prime}$ is an $(n,m,k,\sigma-(\lambda+(m-1)\mu),\chi-m\mu)$-EPDF (or $(n,m,k,\sigma-(\lambda+(m-1)\mu),\sigma-m\mu)$-EPDF).
\end{itemize}
\end{theorem} 
\begin{proof}
\begin{itemize}
    \item[(i)] By definition, for each $D_i$, we have $\Delta(D_i) = \lambda(D_i) + \mu(\rm{G}^*\backslash\{D_i\})$.
    Hence, 
    \begin{equation*}
        \bigcup\limits_{i=1}^{m}\Delta(D_i) = \Delta(D_1) \cup \ldots \cup \Delta(D_m) = (\lambda(D_1) + \mu(G^*\backslash\{D_1\})) \cup \ldots (\lambda(D_m) + \mu(G^*\backslash\{D_m\}))   
    \end{equation*}
\begin{equation*}
        = \lambda(D_1 \cup \ldots \cup D_m) + (m-1)\mu(D_1 \cup \ldots \cup D_m) + m\mu({G}^*\backslash\{D_1 \cup \ldots \cup D_m\})
    \end{equation*}
\begin{equation*}
        = \lambda(S) + (m-1)\mu(S) + m\mu({G}^*\backslash{S}) = (\lambda + (m-1)\mu)(S) + m\mu({G}\backslash{S}). 
    \end{equation*}
    Hence $\mathcal{S}^{\prime}$ forms an $(n,m,k,\lambda+(m-1)\mu,m\mu)$-DPDF. 
    \item[(ii)] This follows from Theorem \ref{thm3}.
\end{itemize}
\end{proof}

\begin{example}
By Theorem \ref{thm5}, if $H$ is a non-trivial subgroup of a group $G$, then $H$ and $H \setminus \{0\}$ are PDSs.  So, if $H_1, \ldots, H_m$ are disjoint equal-sized subgroups of a group $G$ which pairwise intersect only in $\{0\}$, then taking any collection of sets from $\{ H_1 \setminus \{0\}, \ldots, H_r \setminus \{0\} \}$ gives a DPDF.  
\end{example}

Various examples of this set-up exist in the literature.  In a group $G$ of order $n^2$, a collection of order-$n$ subgroups $H_1, \ldots, H_r$ such that $H_i \cap H_j=\{0\}$ for $i \neq j$ is called an $(n,r)$ partial congruence partition of degree $r$ in $G$ (see \cite{JedLi}).  More generally, for any group $G$, a collection of subgroups $H_1, \ldots, H_r$ of $G$ such that the $H_i \setminus \{0\}$ partition $G^*$ is called a partition of $G$ (see \cite{HucPat}).  The groups which admit a subgroup partition have been characterized in \cite{Zap}.  Subgroup partitions were used in \cite{HucPat} to construct EDFs.

We now focus on objects which are simultaneously DPDFs and EPDFs.

\begin{theorem}\label{partition_SEDF}
Let $\mathcal{S'}=\{D_1, \ldots, D_r\}$ be a family of disjoint $k$-subsets in a group $G$ such that each $D_i$ is an $(n,k,\lambda,\mu)$-PDS.  Let $S = \cup_{i=1}^r D_i$.  If any of the following hold:
\begin{itemize}
\item[(i)] $G \setminus S$ is a difference set;
\item[(ii)] $G \setminus S$ is a proper partial difference set;
\item[(iii)] $G^* \setminus S$ is a proper partial difference set;
\end{itemize}
then $\mathcal{S'}$ is a DPDF which is also an EPDF.
\end{theorem}
\begin{proof}
This follows from Theorem \ref{thm5} and Theorem \ref{thm6a}, since in all cases $S$ is either a difference set or a proper partial difference set.
\end{proof}

\begin{example}
\begin{itemize}
\item[(i)] If $H_1, \ldots, H_m$ are equal-sized non-trivial subgroups which form a group partition of $G$, then taking any collection of $m-1$ sets from $\{ H_1 \setminus \{0\}, \ldots, H_m \setminus \{0\} \}$ gives a DPDF which is also an EPDF.  
\item[(ii)] If $H_1, \ldots, H_m$ form a SEDF which partitions $G^*$, then any collection of $m-1$ of these sets form both an DPDF and EPDF.  The only abelian parameters for SEDF partitioning $G^*$ with $m>2$ are $(243,11,22,20)$ (see \cite{JedLi2}).  Taking any 10 of the 11 sets in the $(243,11,22,20)$-SEDF, we obtain a $(243,10,22,19,20)$-DPDF which is also a $(243,10,22,180)$-EDF.
\end{itemize}
\end{example}
 
\begin{example}\label{Z3xZ3}
Let $G=\mathbb{Z}_3 \times \mathbb{Z}_3$.  Let $A_1=\{(1,1), (2,2)\}$, $A_2=\{(0,1), (0,2)\}$, $A_3=\{(1,2), (2,1)\}$ and $A_4=\{(1,0), (2,0)\}$.  These are all additive subgroups with $(0,0)$ removed. Here $A=\{A_1, A_2, A_3, A_4 \}$ forms a $(9,4,2,6)$-EDF which is a $(9,4,2,1)$-DDF.  Theorem \ref{partition_SEDF} guarantees that any choice of $3$ sets from $A$ will form a DPDF/EPDF.  In fact, by direct checking, any choice $B$ of $i$ sets from $A$, where $1 \leq i \leq 3$, forms a DPDF and EPDF.  Note the union of the sets in $B$ is a PDS which is not a DS. 
\end{example}

\section{Developing a cyclotomic framework}

In the rest of the paper, we will obtain constructions for DPDFs and EPDFs via cyclotomy.  For further background on cyclotomic definitions and results, see \cite{Sto}.    Throughout, $q$ will denote a prime power and will be expressed as $q=ef+1$ where $e,f$ are positive integers greater than 1.  Note that $q$ may be even or odd.  The element $\alpha$ will denote a primitive element of $GF(q)$.  

\begin{definition}\label{def3}
Let GF$(q)$ be a finite field of order $q = ef + 1$, where $q$ is an arbitrary prime power. Let $\alpha$ be a primitive element of GF$(q)$. 
\begin{itemize}
\item[(i)] The \emph{cyclotomic classes} of order $e$ of GF$(q)^*$ are defined: 
\begin{center}
$C_i^e = \alpha^i\langle\alpha^e\rangle$   
\end{center}

for $0 \leq i \leq e-1$. Each cyclotomic class is of size $f$.

\item[(ii)]
For each pair of cyclotomic classes $C_i^e$ and $C_j^e$ of order $e$, we define the \emph{cyclotomic number} $(i,j)_e$ of order $e$ to be the number of solutions to the following equation:
\begin{center}
$z_i + 1 = z_j$
\end{center}
where $z_i \in C_i^e$ and $z_j \in C_j^e$.
\end{itemize}
\end{definition}

The following result from literature directly motivates our approach (see e.g. \cite{DavHucMul},\cite{Wil}).
\begin{theorem}
Let $q=ef+1$ be a prime power.  Then the set of all cyclotomic classes $\{C_0^e, \ldots, C_{e-1}^e\}$ is a $(q,e,f, f-1)$-DDF and $(q,e,f,q-1-f)$-EDF.
\end{theorem}

We give the following useful relationship between cyclotomic classes:

\begin{lemma}\label{union}
Let $q=ef+1$. If $\epsilon|e$, then for $0 \leq j \leq \epsilon-1$, $$C_j^{\epsilon}=\bigcup_{i=0}^{\frac{e}{\epsilon}-1} C_{i \epsilon +j}^e.$$
In particular $C_0^{\epsilon}$ is partitioned by the cosets $\{C_{i \epsilon}^e: 0 \leq i \leq \frac{\epsilon}{e}-1 \}$ of $C_0^e$.
\end{lemma}
\begin{proof} It is clear that $C_0^e$ is a subgroup of index $\frac{e}{\epsilon}$ in $C_0^{\epsilon}$.  Hence its cosets in $C_0^{\epsilon}$ partition $C_0^{\epsilon}$, i.e. $C_0^{\epsilon}=\bigcup_{i=0}^{\frac{e}{\epsilon}-1} C_{i \epsilon}^e$.
Then for $0 \leq j \leq \epsilon-1$, $C_j^{\epsilon}=\alpha^j(\bigcup_{i=0}^{\frac{e}{\epsilon}-1} C_{i \epsilon}^e)=\bigcup_{i=0}^{\frac{e}{\epsilon}-1} C_{i \epsilon +j}^e$.
\end{proof}

In what follows, we will often consider the situation when the prime power $q=ef+1$ is also expressible as $q=\epsilon \rho+1$ where $\epsilon|e$ ($\epsilon,\rho>1$).  For many DPDF and EPDF constructions, we will be interested in the case when $e>\epsilon$, but our results will hold for $e=\epsilon$, and this is useful in analysing the PDS situation.  

We will denote the family of sets ${C_0^{\epsilon}}^{\prime}:=\{C_{i \epsilon}^e: 0 \leq i \leq \frac{\epsilon}{e}-1 \}$.  Here, the sets of ${C_0^{\epsilon}}^{\prime}$ are certain cyclotomic classes of order $e$ which partition $C_0^{\epsilon}$, where $C_0^{\epsilon}$ is a cyclotomic class of order $\epsilon$ (of larger cardinality if $e>\epsilon$), where $\epsilon$ is a divisor of $e$.  Note that any coset $\alpha^{i \epsilon} C_0^e$ ($i \in \mathbb{Z}$) represents a set in ${C_0^{\epsilon}}^{\prime}$.

\subsection{Internal differences: PDSs and DPDFs}

We introduce the following notation and results.
\begin{definition}\label{def6CD}
Let $q=ef+1$ and let $\alpha$ be a primitive element of $GF(q)$.
\begin{itemize}
\item[(i)] For each $1 \leq r \leq f-1$,  define 
$$ T_r := \{ \alpha^{ne}-\alpha^{me}:n-m=r \mod f, \, 0 \leq n\neq{m} \leq f-1 \}.$$ \\  
Clearly, $T_r \subseteq \Delta(C_0^e)$; we refer to the set $T_r$ as a \emph{transversal} of $\Delta(C_0^e)$. Note $|T_r|=f$.
\item[(ii)]  For each $1 \leq r \leq f-1$, let $a_r \in \{0,...,e-1\}$ be such that $\alpha^{re}-1 \in C_{a_r}^e = \alpha^{a_r}C_0^e$.
\end{itemize}
\end{definition}

This result summarizes some cyclotomic relationships using our notation.
\begin{lemma}\label{lem1CD}
\begin{itemize}
\item[(i)] For $1 \leq r \leq f-1$, each transversal $T_r = (\alpha^{re} - 1)C_0^e$ is a copy of the cyclotomic class $C_{a_r}^e$ where $a_r \in \{0,...,e-1\}$.

\item[(ii)] $\Delta(C_0^e) = \bigcup\limits^{f-1}_{r=1}{T_r} = \bigcup\limits_{r=1}^{f-1}C_{a_r}^e = \bigcup\limits_{i=0}^{e-1}(i,0)_e(C_i^e)$

\item[(iii)] For $0 \leq j \leq e-1$, $\Delta(C_j^e) = \alpha^j\Delta(C_0^e) = \bigcup\limits^{f-1}_{r=1}\alpha^jT_r = \bigcup\limits^{f-1}_{r=1}\alpha^jC_{a_r}^e = \bigcup\limits_{i=0}^{e-1}(i,0)_e(\alpha^jC_i^e)$.
\end{itemize}
\end{lemma}

From these, together with Theorem \ref{thm6a}, we immediately have:
\begin{lemma}\label{PDS/DPDFConditions}
Let $q=ef+1$ be a prime power.
\begin{itemize}
    \item[(i)] For $0 \leq i \leq e-1$, each $C_i^e$ is a $(q,\frac{q-1}{e},A,B)$-PDS if and only if there exist $A,B$ such that $A=(0,0)_e$ and $B=(i,0)_e$ for all $1 \leq i \leq e-1$.  The PDS is proper precisely when $A \neq B$.
    \item[(ii)] Suppose there exist $A,B$ such that $A=(0,0)_e$ and $B=(i,0)_e$ for all $1 \leq i \leq e-1$.  Let $I \subset \{0,1,\ldots,e-1\}$ ($|I|=u$) and $
    \mathcal{D}^{\prime} = \{C_i^e\}_{i \in I}$.  Then $\mathcal{D}^{\prime}$ is a $(q,u,\frac{q-1}{e},A+(u-1)B,uB)$-DPDF, which is proper precisely when $A \neq B$.
\end{itemize}
\end{lemma}

We make a definition whose motivation will be clear from the subsequent proposition.

\begin{definition}\label{def6CD}
For $1 \leq r \leq f-1$, we define

$$D_r := \bigcup\limits_{i=0}^{e/\epsilon-1} \alpha^{i\epsilon} T_r$$
We will refer to this as a \textit{diagonal of a transversal}.
\end{definition}

\begin{proposition}\label{prop1CD}
\begin{itemize}
\item[(i)] ${\rm Int}({C_0^{\epsilon}}^{\prime}) = \bigcup\limits_{r=1}^{f-1} D_r$.\\
For each $1 \leq r \leq f-1$, 
\begin{itemize}
\item[(ii)] $D_r = (\alpha^{re} - 1)C_0^{\epsilon}$;
\item[(iii)]  $D_r = \alpha^i C_0^{\epsilon}= C_i^{\epsilon}$, where $\alpha^{re}-1 \in C_i^{\epsilon}$ ($0 \leq i \leq \epsilon-1$).
\end{itemize}
\end{itemize}
\end{proposition}

Note that, by (ii) of the above result, our diagonals of transversals of $\Delta(C_0^e)$ are themselves transversals of $\Delta(C_0^{\epsilon})$.

\begin{definition}\label{defg1CD}
Let $q = \epsilon \rho + 1=ef + 1$ where $\epsilon\mid{e}$. Let $\alpha$ be a fixed primitive element of GF$(q)$ and let $C_0^{\epsilon}=\langle \alpha^\epsilon \rangle$. We define, for $0 \leq i \leq \epsilon-1$,
$$\Phi_i := \{x \in C_0^e : x\neq{1}, x-1 \in \alpha^i C_0^{\epsilon}\} \mbox{ and }\phi_i := |\Phi_i|.$$
\end{definition}

The following result establishes the relationship between $\phi_0,\phi_1,\ldots,\phi_{\epsilon-1}$ and cyclotomic numbers.  The quantities $\phi_j$ are hard to evaluate in the general case. 

\begin{proposition}\label{thmg4CD}
Let $q = \epsilon\rho + 1 = ef + 1$ such that $\epsilon \mid e$, then for $0 \leq j \leq \epsilon-1$,
\begin{equation*}
\phi_j = \sum\limits_{i=0}^{e/\epsilon-1}(\epsilon{i}+j,0)_e.
\end{equation*}
\end{proposition}

We have the following conditions for DDFs and DPDFs.
\begin{theorem}\label{prop2CD}
\begin{itemize}
\item[(i)] If $\phi_1 = \ldots = \phi_{\epsilon-1}$, then $\rm{Int}({C_0^{\epsilon}}^{\prime}) = \phi_0(C_0^{\epsilon}) + \phi_1(\rm{G}^*\backslash{C_0^{\epsilon}})$, i.e. ${C_0^{\epsilon}}^{\prime}$ is a DPDF (or PDS when $\epsilon=e$).
\item[(ii)] If $\phi_0 = \phi_1 = \ldots = \phi_{\epsilon-1}$, then $\rm{Int}({C_0^{\epsilon}}^{\prime}) = \phi_0(\rm{G}^*)$, i.e. ${C_0^{\epsilon}}^{\prime}$ is a DDF (or DS when $\epsilon=e$).
\end{itemize}
\end{theorem}

We now look at a useful relationship between transversals and diagonals of transversals:

\begin{proposition}\label{lem5CD}
Let $q=ef+1$. For each $1 \leq r \leq f-1$, 
\begin{itemize}
\item[(i)] $T_{f-r} = -T_r$
\item[(ii)] $D_{f-r} = -D_r$
\end{itemize}
\end{proposition}

\begin{proof}
\begin{itemize}
\item[(i)] By Lemma \ref{lem1CD}(i), $-T_r = -(\alpha^{re}-1)C_0^e  = (1-\alpha^{re})C_0^e$ for each $1 \leq r \leq f-1$. Since $\alpha^{ef}=1$, $-T_r = (\alpha^{ef}-\alpha^{re})C_0^e = \alpha^{re}(\alpha^{e(f-r)}-1)C_0^e = (\alpha^{e(f-r)}-1)C_0^e = T_{f-r}$.
\item[(ii)] Since by Proposition \ref{prop1CD}(ii), $D_r = (\alpha^{re}-1)C_0^{\epsilon}$, the proof is analogous to part (i).
\end{itemize}
\end{proof}

Since $-T_r=(-1)T_r$ and $-C_0^{\epsilon}=(-1)C_0^{\epsilon}$, we may exploit this relationship between a transversal and its negative by determining the cyclotomic class containing the element $-1$.  Note that, for an odd prime power $q=\epsilon \rho+1$, $q \equiv 1 \mod 2 \epsilon$ precisely when $\rho$ is even, and  $q \equiv \epsilon + 1 \mod 2 \epsilon$ precisely when $\rho$ is odd; moreover when $\epsilon$ is even and $\rho$ is odd, $e$ must be an odd multiple of $\epsilon$.  The following result summarises results from \cite{Sto}:

\begin{lemma}\label{lem4CD}
Let $q=\epsilon \rho+1$.
\begin{itemize}
\item[(a)]
Let $\epsilon$ be an even integer.  Then
\begin{itemize}
\item[(i)] when $q \equiv 1 \mod 2\epsilon$, $-1 \in C_0^{\epsilon}$.
\item[(ii)] when $q \equiv \epsilon + 1 \mod 2\epsilon$, $-1 \in \alpha^{\frac{\epsilon}{2}}C_0^{\epsilon}$.
\end{itemize}
\item[(b)] Let $\epsilon$ be an odd integer. Then, $-1 \in C_0^{\epsilon}$.
\end{itemize}
\end{lemma}

We now identify precisely how $D_r$ and $D_{f-r}$ are related.  Note that, when $f$ is even, $T_{\frac{f}{2}}$ and $D_{\frac{f}{2}}$ occur.  These form their own negatives; as such, both will be treated separately to other transversals/diagonals of transversals in later counting arguments. 

\begin{proposition}\label{diagonals}
\begin{itemize}
\item[(a)]
Let $D_r=\alpha^i C_0^{\epsilon}$, where $1 \leq r \leq f-1$ ($r \neq \frac{f}{2}$) and $0 \leq i \leq \epsilon-1$.
\begin{itemize}
\item[(i)]
Let $q=\epsilon \rho+1$ and let $\epsilon$ be an even integer.  Then
\begin{itemize}
\item[(I)] when $q \equiv 1 \mod 2\epsilon$, $D_{f-r}=\alpha^i C_0^{\epsilon}$.
\item[(II)] when $q \equiv \epsilon + 1 \mod 2\epsilon$, $D_{f-r}= \alpha^{i+\frac{\epsilon}{2}}C_0^{\epsilon}$.
\end{itemize}
\item[(ii)] Let $q = \epsilon\rho + 1$, and let $\epsilon$ be an odd integer. Then, $D_{f-r}= \alpha^i C_0^{\epsilon}$. 
\end{itemize}
\item[(b)] If $f$ is even, $T_{\frac{f}{2}}=(-2)C_0^e$ and $D_{\frac{f}{2}}=(-2)C_0^{\epsilon}$.
\end{itemize}
\end{proposition}

In the case when $\epsilon=2$, we have the following useful result.

\begin{lemma}\label{central}
Let GF$(q)$ be a finite field $q=ef+1$, with $e$ even.  Let $f$ be even. 
\begin{itemize}
    \item[(i)] If $q \equiv 1 \mod 8$, $D_{\frac{f}{2}}=C_0^2$.
    \item[(ii)] If $q \equiv 5 \mod 8$, $D_{\frac{f}{2}} = C_1^2$.
\end{itemize}
\end{lemma}
\begin{proof}
It is known (see \cite{Sto}) that for a finite field $GF(q)$ of odd order,  if $q \equiv 1,7 \mod 8$ then $2 \in C_0^2$, while if $q \equiv 3,5 \mod 8$ then $2 \in C_1^2$.  The result follows by combining this with Proposition \ref{diagonals} and Lemma \ref{lem4CD}.
\end{proof}

Finally, we prove the following consequences of the structural results we have established, which will be useful in what follows.  For $0 \leq i \leq \epsilon-1$ we define
$$ \Psi_i := \{1 \leq r < \frac{f}{2}: \alpha^{re} \in \Phi_i \} \mbox{ and } \psi_i := |\Psi_i|.$$

\begin{theorem}\label{combined}
Let $q=ef+1=\epsilon \rho+1$ ($\epsilon|e$).
\begin{itemize}
\item[(a)] Let $q\equiv 1 \mod 2\epsilon$ (i.e. $\rho$ is even).
Suppose that $\phi_1=\phi_j$ for all $1 \leq j \leq \epsilon-1$, i.e. ${C_0^{\epsilon}}^{\prime}$ is a DPDF (or PDS when $e=\epsilon$). 
\begin{itemize}
\item[(i)] If $f$ is odd then $\phi_0=2 \psi_0$ and $\phi_1=2 \psi_1$.
\item[(ii)] If $f$ is even and $\epsilon>2$ then $\phi_0=2 \psi_0+1$ and $\phi_1=2 \psi_1$. 
\item[(iii)] If $f$ is even, $\epsilon=2$ and $q \equiv 1 \mod 8$ then $\phi_0=2 \psi_0+1$ and $\phi_1=2 \psi_1$. 
\item[(iv)] If $f$ is even, $\epsilon=2$ and $q \equiv 5 \mod 8$ then $\phi_0=2\psi_0$ and $\phi_1=2 \psi_1+1$.
\end{itemize}
\item[(b)] Let $q\equiv  \epsilon+1 \mod 2\epsilon$ (i.e. $\rho$ is odd).  Let $\epsilon$ be odd.  Suppose that $\phi_1=\phi_j$ for all $1 \leq j \leq \epsilon-1$.  Then $\phi_0=2 \psi_0$ and $\phi_1=2\psi_1$.
\item[(c)] Let $q\equiv \epsilon+1 \mod 2\epsilon$.  If $\epsilon$ is even, then $\phi_0=\phi_{\frac{\epsilon}{2}}$.
\end{itemize}
\end{theorem}
\begin{proof}
\begin{itemize}
\item[(a)] Since  $q\equiv 1 \mod 2\epsilon$, we have $\rho$ even.  If $f \neq \rho$ then $f$ can be odd or even.\\
Suppose $f$ is odd.  Consider the diagonals of transversals of $\Delta(C_0^e)$.  There is no central diagonal $D_{\frac{f}{2}}$.  By Proposition \ref{diagonals}, $D_r=D_{f-r}$.  Hence if $r \in \Psi_i$, i.e. $\alpha^{re} \in \Phi_i$ for some $i$ then also $\alpha^{(f-r)e} \in \Phi_i$, i.e. the $f-1$ diagonals of ${\rm Int}({C_0^{\epsilon}}^{\prime})$ pair up into $\frac{f-1}{2}$ disjoint pairs, each pair contributing two copies of a given $\alpha^i C_0^{\epsilon}$.  Hence $\phi_0=2 \psi_0$ and $\phi_1=2\psi_1$.\\
Now suppose $f$ is even (note this case holds when $e=\epsilon$ and $f=\rho$).  There is one central diagonal $D_{\frac{f}{2}}$ and $\frac{f-2}{2}$ pairs of diagonals which pair up according to the rule $D_r=D_{f-r}$.  Finally, consider $D_{\frac{f}{2}}$.  For $\epsilon>2$, if $\alpha^{\frac{f}{2}e} \in \Phi_i$ where $i \in \{0, \ldots, \epsilon-1\}$, i.e. $D_{\frac{f}{2}}=(\alpha^{\frac{f}{2}e}-1)C_0^{\epsilon}=\alpha^i C_0^{\epsilon}$, then $\phi_i=2 \psi_i+1$, while $\phi_j= 2 \psi_j$ for all $j \neq i$.  If $i \neq 0$, this contradicts the assumption that $\phi_1=\phi_j$ for all $1 \leq j \leq \epsilon-1$, hence $i=0$.  For $\epsilon=2$, this argument does not apply; instead we apply Lemma \ref{central}.
\item[(b)] Similar to (a).
\item[(c)] Since  $q\equiv \epsilon+1 \mod 2\epsilon$, we have $\rho$ odd.  Since $\rho=\frac{e}{\epsilon}f$,  $\epsilon|e$ and $\rho$ is odd, $f$ must be odd.  Consider the diagonals of transversals of $\Delta(C_0^e)$.  Since $f$ is odd, there is no central diagonal $D_{\frac{f}{2}}$.  By Proposition \ref{diagonals}, if $D_r=\alpha^i C_0^{\epsilon}$ then $D_{f-r}=\alpha^{i+\frac{\epsilon}{2}} C_0^{\epsilon}$.  Hence the $f-1$ diagonals of ${\rm Int}({C_0^{\epsilon}}^{\prime})$ pair up into $\frac{f-1}{2}$ disjoint pairs, and there are the same number of $\alpha^{re} \in \Phi_i$ as there are in $\Phi_{i+\frac{\epsilon}{2}}$.  Hence, for all $0 \leq i \leq \epsilon-1$, $\phi_i=\phi_{i+\frac{\epsilon}{2}}$, so in particular $\phi_0=\phi_{\frac{\epsilon}{2}}$.
\end{itemize}
\end{proof}

\begin{corollary}\label{combined_cor}
Let $q=ef+1=\epsilon \rho+1$ ($\epsilon|e$).
\begin{itemize}
\item[(a)] Let $q\equiv 1 \mod 2\epsilon$.
\begin{itemize}
\item[(i)] If $C_0^{\epsilon}$ is a PDS, then it must be proper (i.e. $C_0^{\epsilon}$ cannot be a difference set).
\item[(ii)] Let $f$ be even.  If ${C_0^{\epsilon}}^{\prime}$ is a DPDF, then it must be proper.
\end{itemize}
Let $\epsilon>2$.
\begin{itemize}
\item[(iii)] If $C_0^{\epsilon}$ is a (proper) PDS then $-2, 2 \in C_0^{\epsilon}$.
\item[(iv)] Let $f$ be even.  If ${C_0^{\epsilon}}^{\prime}$ is a (proper) DPDF then $-2, 2 \in C_0^{\epsilon}$.
\end{itemize}
\item[(b)] Let $q\equiv \epsilon+1 \mod 2\epsilon$ and suppose $\epsilon$ is even.
\begin{itemize}
\item[(i)] $C_0^{\epsilon}$ cannot be a proper PDS.
\item[(ii)] If $\epsilon<e$, then ${C_0^{\epsilon}}^{\prime}$ cannot be a proper DPDF.
\end{itemize}
\end{itemize}
\end{corollary}

\begin{proof}
\begin{itemize}
\item[(a)] We apply Theorem \ref{combined} (a).  For $\epsilon>2$, by assumption,  in both the PDS and DPDF cases, we have $\phi_1=\phi_j$ ($1 \leq j \leq \epsilon-1$).  For (i), we take $e=\epsilon$ and $f=\rho$; as $\rho$ is even,  $\phi_0$ is of opposite parity to $\phi_1$, so $C_0^{\epsilon}$ is not a difference set.  Part (ii) follows similarly by taking $e>\epsilon$ in the lemma.  For $\epsilon=2$ and $f$ even, it is also the case that $\phi_0$ and $\phi_1$ have opposite parity.  For (iii) and (iv), from the proof of Theorem \ref{combined}(a)(ii), we have that $-1=\alpha^{\frac{f}{2}e}\in \Phi_0$, so $-2 \in C_0^{\epsilon}$, and hence $2 \in C_0^{\epsilon}$.
\item[(b)] Apply Theorem \ref{combined}(c) with $e=\epsilon$ for (i) and $e>\epsilon$ for part (ii).  Since $\phi_0=\phi_{\frac{\epsilon}{2}}$ and $\frac{\epsilon}{2} \in \{1, \ldots, \epsilon\}$, it is impossible for $\phi_0$ to take a distinct value from all other $\phi_j$, $j \in \{1, \ldots, \epsilon\}$.
\end{itemize}
\end{proof}

In (a), the $\epsilon=2$ case is different to the $\epsilon>2$ case due to the absence of $\phi_j$ with $2 \leq j \leq \epsilon-1$.  We note that parts (iii) and (iv) of Corollary \ref{combined_cor}(a) do not hold for $\epsilon=2$.

\subsection{External differences: EPDFs}

For the external case, we may make analogous definitions to those of the previous subsection.
\begin{definition}\label{ExternalDef}
Let $q=ef+1$ and let $\alpha$ be a primitive element of $GF(q)$.
\begin{itemize}
\item[(i)] For each $1 \leq r \leq f-1$ and $1 \leq j \leq e-1$,  define 
$$ T_{(r,j)} := \{ \alpha^{ne+j}-\alpha^{me}:n-m=r \mod f, \, 0 \leq n,m \leq f-1 \}.$$ \\  
Clearly, $T_{(r,j)} \subseteq \Delta(C_j^e, C_0^e)$; we refer to the set $T_{(r,j)}$ as an \emph{external transversal} of $\Delta(C_j^e, C_0^e)$. Note $|T_{(r,j)}|=f$.
\item[(ii)]  For each $1 \leq r \leq f-1$ and $1 \leq j \leq e-1$, let $a_{(r,j)} \in \{0,...,e-1\}$ be such that $\alpha^{re+j}-1 \in C_{a_{(r,j)}}^e = \alpha^{a_{(r,j)}}C_0^e$.
\end{itemize}
\end{definition}

This result summarizes some cyclotomic relationships using our notation.
\begin{lemma}\label{ExternalLem}
\begin{itemize}
\item[(i)] For $1 \leq r \leq f-1$ and $1 \leq j \leq e-1$, each external transversal $T_{(r,j)} = (\alpha^{re+j}-1)C_0^e$ is a copy of the cyclotomic class $C_{a_{(r,j)}}$.
    \item[(ii)] $\Delta(C_j^e,C_0^e) = \bigcup\limits_{r=1}^{f}T_{(r,j)} = \bigcup\limits_{r=1}^{f}C_{a_{(r,j)}}^e = \bigcup\limits_{i=0}^{e-1}(i,j)_e(C_i^e)$. 
    \item[(iii)] For $0 \leq l \leq e-1$, $\Delta(C_{j+l}^e,C_l^e) = \alpha^l\Delta(C_j^e,C_0^e) = \bigcup\limits_{r=1}^{f}\alpha^l T_{(r,j)} = \bigcup\limits_{r=1}^{f}\alpha^lC_{a_{(r,j)}}^e = \bigcup\limits_{i=0}^{e-1}(i,j)_e(\alpha^lC_i^e)$.
\end{itemize}
\end{lemma}

\begin{proposition}\label{EPDFLem}
Let $\rm{GF}(q)$ be a finite field, where $q$ is a prime power and $e \geq 3$ is a divisor of $q-1$.
Let $I \subset \{0,1,\ldots,e-1\}$ ($|I|=u,  2\leq u \leq e-1$) and $\mathcal{D}^{\prime} = \{C_i^e\}_{i \in I}$.  \\
If there exist $B, X$ such that $B = (0,i)_e = (i,i)_e$ and $X=(i,j)_e$ for all $1 \leq i\neq{j} \leq e-1$, then
\begin{itemize}
\item[(i)]  $\Delta(C_s^e,C_t^e) = B (C_s^e \cup C_t^e) + X(G^*\setminus  (C_s^e \cup C_t^e))$ for all $0 \leq s\neq{t} \leq e-1$;
\item[(ii)]  $\mathcal{D}^{\prime}$ is a $(q,u,\frac{q-1}{e},2B(u-1)+X(u-1)(u-2),Xu(u-1))$-EPDF (proper when $B \neq X$). 
\end{itemize}
\end{proposition}

\begin{proof}
For (i), apply Lemma \ref{ExternalLem}; note $\Delta(C_s^e,C_t^e) =\alpha^t\Delta(C_j^e,C_0^e)$ where $j  \equiv s-t \mod e$.

For (ii),  the multiset ${\rm Ext}(\mathcal{D}^{\prime}) = \bigcup\limits_{l \neq{j} ; l,j \in I} \Delta(C_l^e,C_j^e)$ is the union of $u(u-1)$ multisets of the form $\Delta(C_l^e,C_j^e)$ (since $|I| = u$).  For a fixed $l \in I$, there are $u-1$ multisets of the form $\Delta(C_l^e,C_j^e)$ in ${\rm Ext}(\mathcal{D}^{\prime})$, where $j \in I$ and $j \neq l$, and $u-1$ multisets of the form $\Delta(C_j^e,C_l^e)$. There are  $u(u-1) - 2(u-1) = (u-1)(u-2)$ multisets of the form $\Delta(C_s^e,C_j^e)$, where $s\neq{j}\neq{l} \in I$. Applying (i), we see that $C_l^e$ occurs precisely $2B(u-1) + X(u-1)(u-2)$ times in the multiset ${\rm Ext}(\mathcal{D}^{\prime})$. Finally, again by (i), for any $t \not\in I$,  $C_t^e$ occurs precisely $u(u-1)X$ times in ${\rm Ext}(\mathcal{D}^{\prime}) = \bigcup\limits_{l\neq{j} : l,j \in I} \Delta(C_l^e,C_j^e)$. The EPDF result follows. 
\end{proof}

\subsection{A combined approach: uniform cyclotomy}

Observe that, by combining the conditions in Lemma \ref{PDS/DPDFConditions} and Proposition \ref{EPDFLem}, we could obtain a class of objects which are both DPDFs and EPDFs.  The following useful definition was introduced by Baumert, Mills and Ward in \cite{BaMiWa}.

\begin{definition}\label{Baumert}
The cyclotomic numbers $(i,j)_e$ over $GF(q)$ are \emph{uniform} if $(0,i)_e=(i,0)_e=(i,i)_e=(0,1)_e$ for $i \neq 0$, and $(i,j)_e=(1,2)$ for $0 \neq i \neq j \neq 0$.
\end{definition}

So, if the cyclotomic numbers are uniform, we would have a family of DPDFs which are also EPDFs.  The following result from \cite{BaMiWa} is important here.

\begin{theorem}\label{MillsThm}
Let $\rm{GF}(q)$ be a finite field where $q$ is a power of a prime $p$ and let $e \geq 3$ be a divisor of $q-1$.  The cyclotomic numbers of order $e$ over ${\rm GF}(q)$ are uniform if and only if $-1$ is a power of $p$ modulo $e$. If this holds, then either $p=2$ or $f = \frac{q-1}{e}$ is even; $q = s^2$ with $s \equiv 1 \mod e$; and 
\begin{equation*}
    (0,0)_e = \left(\frac{s-1}{e}\right)^2 - (e-3)\left(\frac{s-1}{e}\right) -1
\end{equation*}
\begin{equation*}
    (0,i)_e = (i,0)_e = (i,i)_e = \left(\frac{s-1}{e}\right)^2 + \left(\frac{s-1}{e}\right) \,\rm{for}\,i\neq{0},
\end{equation*}
\begin{equation*}
    (i,j)_e = \left(\frac{s-1}{e}\right)^2 \,\rm{for}\, 0 \neq i \neq j.
\end{equation*}
\end{theorem}

We present the following condition which is equivalent to the key condition in Theorem \ref{MillsThm}:

\begin{lemma}\label{Equiv_Baumert}
Let  $q^{\prime}=p^n$ where $p$ is prime, and let $e \geq 3$ be a divisor of $q^{\prime}-1$. Then the following two conditions are equivalent:
\begin{itemize}
\item[(i)] $-1$ is a power of $p$ modulo $e$ 
\item[(ii)] there exists a prime power $q$ such that $q^{\prime}=q^{2b}$ ($b \in \mathbb{N}$) and $e|q+1$.
\end{itemize}
\end{lemma}
\begin{proof}
For the forward direction: let $s$ be the smallest positive integer such that $p^s \equiv -1 \mod e$.  Then $2s|n$ and so $n=2sb$ for some $b \in \mathbb{N}$, i.e. $q^{\prime}=p^n=p^{2sb}=(p^s)^{2b}$, and $e|p^s+1$.  The reverse direction is immediate as $q=p^a$ for some $a \in \mathbb{N}$.
\end{proof}

This leads to the following result.  Observe that, as a special case of part (iii) (taking $e=q+1$), we obtain the result of Calderbank and Kantor in Section 9 of \cite{CalKan} which is presented in different notation in Section 10 of \cite{Ma}.

\begin{theorem}\label{ParamThm2}
Let ${\rm GF}(q')$ be a finite field of order $q^{\prime} = q^{2\beta}$, where $\beta \in \mathbb{N}$ and $q$ is a power of a prime p. Let $e\mid{q+1}$ and set $\eta = \left(\frac{(-q)^{\beta}-1}{e}\right)$. For any $I \subset \{0,1,\ldots,e-1\}$, ($|I|=u, 2 \leq u \leq e-1$) where $\mathcal{D}^{\prime} = \{C_i^e\}_{i \in I}$, 
\begin{itemize}
\item[(i)] each $C_i^{e}$ is a (regular) $(q^{\prime},\frac{q^{\prime}-1}{e},\eta^2-(e-3)\eta-1,\eta^2+\eta)$-PDS;
\item[(ii)] $\mathcal{D}^{\prime}$ is a $(q^{\prime},u,\frac{q'-1}{e},u\eta^2+(u+2-e)\eta-1,u(\eta^2+\eta))$-DPDF and a $(q^{\prime},u,\frac{q'-1}{e},u(u-1)\eta^2+2(u-1)\eta,u(u-1)\eta^2)$-EPDF.
\item[(iii)]  $D=\cup_{i \in I} C_i^e$ is a (regular) $(q^{\prime}, \frac{u(q^{\prime}-1)}{e}, u^2 \eta^2+(3u-e)\eta -1, u^2 \eta^2+u \eta)$-PDS, which is proper except when $\eta=1=2u-e$ or $\eta=-1=2u-e$.
\end{itemize}
\end{theorem} 

\begin{proof}
By Lemma \ref{Equiv_Baumert}, the conditions of Theorem \ref{MillsThm} are satisfied, and so the cyclotomic numbers are uniform with the given values, where (in the notation of Theorem  \ref{MillsThm}) we set $\eta=\frac{s-1}{e}$.   To determine the specific value of $\eta$: from Theorem \ref{MillsThm},  $q^{\prime}=s^2$ with $s \equiv 1 \mod e$, and as $q^{\prime} = q^{2\beta}$, it is clear that $s \in  \{\pm {q}^{\beta} \}$.  Since $-q \equiv 1 \mod e$,  we see that $(-q)^{\beta}$ is always congruent to 1 modulo $e$. Hence,  we may set $\eta = \frac{s-1}{e} = \frac{(-q)^{\beta}-1}{e}$.   Parts (i) and (ii) then follow: Lemma \ref{PDS/DPDFConditions} then gives the PDS and DPDF results, while Proposition \ref{EPDFLem} yields the EPDF result.  In (i), analysis of parameters shows that $C_i^e$ is a proper PDS except in the trivial case when $q=4$, hence regular; in (ii), the DPDF is proper since $e \neq 2$ and the EPDF is proper since $u \neq 1$.  For (iii), combine (ii) with Theorem \ref{thm4}. Equating the two frequencies, $D$ is a difference set precisely when $\eta(2u-e)=1$.  This leads to the two stated cases; in both cases $q=4u^2$, i.e. $p=2$ (so $D$ is always regular).  
\end{proof}

\begin{remark}\label{Hadamard}
When $D$ in Theorem \ref{ParamThm2}(iii) is a DS, the first case forces $e=3$ and $\beta=2$, i.e. $D$ must be a $(16,10,6)$-DS, while in the second case, $D$ is a Hadamard Difference Set in $GF(2^{2 \beta})$ with parameters $(4u^2, 2u^2-u, u^2-u)$ where $u=2^{\beta-1}$, $\beta \in \mathbb{N}$.    This was noted and proved in Section 8 of \cite{BaMiWa}. For the first case, take $\mathcal{D}^{\prime}=\{C_i^3,C_j^3\}$ in $GF(16)$ where $i \neq j$;  for the second, take $\mathcal{D}^{\prime}=\{C_i^5, C_j^5\}$ in $GF(16)$ where $i \neq j$.
\end{remark}

\begin{example}
Let $q^{\prime}=3^6=729$ in Theorem \ref{ParamThm2}; here the possible values of $q$ are $3$ and $3^3=27$.  For $q=3$, we may take $e=4$, while for $q=27$ we may take $e=4,7,14,28$.  In the case when $q=3$ and $e=4$, $\eta=\frac{(-3)^3-1}{4}=-7$; by  Theorem \ref{ParamThm2}(i), each $C_i^4$ is a $(729,182,69,42)$-PDS, which may be used to form DPDFs, EPDFs and PDSs by parts (ii) and (iii).  When $q=27$ and $e=14$, $\eta=\frac{(-27)^1-1}{14}=-2$, and each $C_i^{14}$ is a $(729,52,25,2)$-PDS.
\end{example}

\section{Cyclotomic PDSs and applications to DPDFs/EPDFs}

In this paper, we will use cyclotomic PDSs in two distinct ways to construct DPDFs and EPDFs: as the constituent sets, or by partitioning.  Both of these require a knowledge of when $C_0^e$ forms a partial difference set (or difference set).  Throughout, $q$ is a prime power which we express as $q=ef+1$.

\subsection{When is $C_0^e$ a DS or PDS?}

It is possible to determine the conditions under which $C_0^e$ is a DS or PDS for specific values of small $e$ by using cyclotomic numbers directly.  In general, explicit evaluation of cyclotomic numbers is a difficult problem.  Selected values for $e=2,3,4,6$ and $8$ are given in \cite{Sto}; further results have been  obtained for $e \leq 12$ (see \cite{Leh} and \cite{Dic}) and certain values up to $e=24$ (see \cite{BeEvWi}). 

For difference sets, various results are given in \cite{Sto}.  The recent paper \cite{Xia} describes current progress on the question of when $C_0^e$ forms a difference set and establishes new results, including the case when $q$ is even.  The key results from \cite{Sto} and \cite{Xia} relevant to this paper are given below.  Note that for $q=ef+1$, $C_0^e$ is not a difference set if $e$ is odd.

\begin{theorem}\label{DSresults}
Let $q=ef+1$ be a prime power.
\begin{itemize}
\item[(i)] If $q$ is even then $C_0^e$ is not a difference set for any value of $e$.
\item[(ii)] If $q$ is odd and $C_0^e$ is a difference set then $e$ is even and $f$ is odd.
\item[(iii)] If $q=2f+1 \equiv 3 \mod 4$ then $C_0^2$ is a difference set.
\item[(iv)] If $q=4f+1$ then $C_0^4$ is a difference set if and only if $q=1+4t^2$ and $t$ is odd (here $f$ is odd).
\item[(v)] If $q=6f+1$, there is no $C_0^6$ which is a difference set. 
\item[(vi)] If $q=8f+1$ then $C_0^8$ is a difference set if and only if $q$ admits the simultaneous representations $q=9+64y^2=1+8b^2$ where $y \equiv b \equiv 1 (\mod 2)$ (here $f$ is odd).
\end{itemize}
\end{theorem}

For partial difference sets, recall that by Corollary \ref{combined_cor}, for $q=ef+1$ with $e$ even, $C_0^e$ is not a proper $(q,f,\lambda,\mu)$-PDS if $f$ is odd.  The $e=2$ case is well-known (\cite{Sto}):
\begin{proposition}\label{prop1}
Let $q=2f+1$ be an odd prime power.
Then $C_0^2$ forms a $(q, \frac{q-1}{2}, \frac{q-5}{4}, \frac{q-1}{4})$-PDS if $q \equiv 1 \mod 4$ (and a $(q, \frac{q-1}{2}, \frac{q-3}{4})$-DS if $q \equiv 3 \mod 4$) in $GF(q)$.
\end{proposition}

We establish a comparable result for other small values of $e$. 

\begin{theorem}\label{PDSparams}
Let $GF(q)$ be a finite field, where $q=ef+1$ is a prime power.
\begin{itemize}
\item[(i)] If $q = 3f + 1$, such that $4q = c^2 + 27d^2$ is the proper representation of $q$ with $c \equiv 1 \mod 3$, then $C_0^3$ is a PDS if and only if $d=0$.\\  
Here the parameters are $(q,\frac{q-1}{3},\frac{q-8+c}{9},\frac{2q-4-c}{18})$ and $C_0^3$ is proper if non-trivial. This holds precisely when $q = p^m$, such that $p \equiv 2 \mod 3$ and $m$ is even.  
\item[(ii)] If $q = 4f+1$, such that $q = s^2+4t^2$ is the proper representation of $q$ with $s \equiv 1 \mod 4$, then $C_0^4$ is a PDS if and only if $t=0$. \\
 Here the parameters are $(q,\frac{q-1}{4},\frac{q-11-6s}{16},\frac{q-3+2s}{16})$. This holds precisely when $q = p^m$, where $p \equiv 3 \mod 4$ and $m$ is even.
\item[(iii)] If $q = 6f + 1$ such that $q = s^2 + 3t^2$ is the proper representation of $q$ with $s \equiv 1 \mod 3$, then $C_0^6$ is a PDS if and only if $t=0$.\\
Here the parameters are $(q,\frac{q-1}{6},\frac{q-17-20s}{36},\frac{q-5+4s}{36})$. This holds precisely when $q = p^m$ where $p \equiv 5 \mod 6$ and $m$ is even.
\item[(iv)] If $q = 8f+1$ such that $q = x^2 + 4y^2 = a^2 + 2b^2$ are proper representations of $q$ with $x \equiv a \equiv 1 \mod 4$, then $C_0^8$ is a PDS if and only if $x=a$ and $y=b=0$.\\
Here the parameters are $(q,\frac{q-1}{8},\frac{q-23-42x}{64},\frac{q-7+6x}{64})$. This holds precisely when $q=p^m$, where $p \equiv 7 \mod 8$ and $m$ is even.
\end{itemize}
\end{theorem}
\begin{proof}
\begin{itemize}
\item[(i)] For the forward direction, suppose that $C_0^3$ is a $(q, \frac{q-1}{3}, A, B)$-PDS.  So $A = (0,0)_3$ and $B = (1,0)_3=(2,0)_3$.  By Theorem \ref{thm:e=3}, the latter equality happens when 
\begin{center}
$(1,0) = \frac{2q-4-c-9d}{18} = \frac{2q-4-c+9d}{18} = (2,0)$
\end{center} i.e. precisely when $d=0$.   For the reverse direction: by Theorem \ref{thm:e=3}, since $d=0$, we have $p \equiv 2 \mod 3$. This means that $-1 \equiv p \mod e$, and so by Lemma \ref{Equiv_Baumert} and Theorem \ref{ParamThm2}, this means that $C_0^3$ is a PDS.  For the parameters, apply Theorem 7.1 with  $d=0$ to see that $C_0^3$ is $(q,\frac{q-1}{3},\frac{q-8+c}{9},\frac{2q-4-c}{18})$-PDS. Notice that the only value of $c \equiv 1 \mod 3$ for which $A=\frac{q-8+c}{9}=\frac{2q-4-c}{18}=B$ is $c=4$, which is the trivial case $q=4$. 
\item[(ii)] Similar to (i); this result was stated without proof in \cite{Ma}.
\item[(iii)] Similar to (i).  By Corollary \ref{combined_cor}, when $C_0^6$ is a PDS we must have $2 \in C_0^6$. For the forward direction, the appropriate case of Theorem \ref{thm:e=6} may then be applied and the equation solved to see $t=0$.  For the reverse direction, by Theorem \ref{thm:e=6},  we have $p \equiv 5 \mod 6$ when $t=0$ and so Lemma \ref{Equiv_Baumert} and Theorem \ref{ParamThm2} apply to show $C_0^6$ is a PDS.  For the parameters, we again use Theorem \ref{thm:e=6}.
\item[(iv)] Similar to (i).  As in (iii), we apply Corollary \ref{combined_cor} to see that $2 \in C_0^8$, which determines the appropriate case in Theorem \ref{thm:e=8}.  For the reverse direction, observe that by Theorem \ref{thm:e=8}, when $y=0$, $p \equiv 3 \mod 4$, and when $b=0$, $p \equiv 5 \mod 8$ or $p \equiv 7 \mod 8$. To satisfy both of these equations simultaneously (as $y=b=0$), $p \equiv 7 \mod 8$, meaning $-1 \equiv p \mod e$.   The rest of the proof then follows as above.
\end{itemize}
\end{proof}

We end this section with two results (not derived from cyclotomic numbers), which give a characterization of $C_0^e$ in terms of the relationship between $e$ and $f$.

\begin{proposition}\label{C0_PDS_prop}
\begin{itemize}
    \item[(i)] Let $q=p^{\beta}$, $\beta >1$.  For $r|\beta$ ($r \neq \beta$), let $F_r$ be the subfield $GF(p^r)$ of $GF(q)$; then $F_r \setminus \{0\}$ is the cyclotomic class $C_0^e$ of $GF(q)$ where $e=\frac{q-1}{p^r-1}$.
    \item[(ii)]  Let $q=ef+1$.  Then $C_0^e$ is a $(q,f,f-1,0)$-PDS if and only if $C_0^e \cup \{0\}$ is a subfield of $GF(q)$.  Here, $q$ is a prime power which is not a prime, $e=\frac{q-1}{p^r-1}$ and $f=p^r-1$.
\end{itemize}
\end{proposition}
\begin{proof}
(i) Since $F_r$ is a subfield of order $p^r$, $F_r\backslash\{0\}$ is a multiplicative subgroup of GF$(q)^*$ with cardinality $p^r-1$. Each subgroup of the cyclic group \rm{GF}$(q)^*$ of given order is unique up to isomorphism. Hence $F_r\backslash\{0\} = C_0^e$, where $e = \frac{q-1}{p^r-1}$. \\
(ii) Suppose $C_0^e$ is a  $(q,f,f-1,0)$-PDS.  By Theorem \ref{thm5}, $C_0^e \cup \{0\}$ is an additive subgroup of $GF(q)$ and by definition $C_0^e$ is a multiplicative sugroup of $GF(q)^*$, so $C_0^e \cup \{0\}$ is a subfield.  Conversely, suppose $C_0^e \cup \{0\}$ is a subfield and hence an additive subgroup of $GF(q)$; then by Theorem \ref{thm5}, it is a $(q,f+1,f+1,0)$-PDS.  By subfield properties, $C_0^e \cup \{0\}$ and $C_0^e$ are closed under the taking of negatives.  Hence removal of $\{0\}$ yields a $(q,f,f-1,0)$-PDS. The parameters $e$ and $f$ are as in part (i). 
\end{proof}

 In fact the subfield situation is the only possibility unless $e<f$.
 \begin{theorem}\label{eandf}
 Let $q=p^{\beta}=ef+1$.  Suppose $C_0^e$ is a proper $(q,f,\lambda,\mu)$-PDS.  Then
 \begin{itemize}
  \item[(i)] if $e>f$, $C_0^e \cup \{0\}$ is a subfield of $GF(q)$;
     \item[(ii)] if $e<f$ then $C_0^e \cup \{0\}$ is not a subfield of $GF(q)$ and $\mu \geq 1$;
     \item[(iii)] the case $e=f$ cannot occur for $f>2$.
 \end{itemize}
 \end{theorem}
 \begin{proof}
As $C_0^e$ is a $(q,f,\lambda,\mu)$-PDS then the following multiset equation must hold;
    \begin{equation*}
        \Delta(C_0^e) = \lambda(C_0^e) + \mu(G^*\backslash{C_0^e}).
    \end{equation*}
The set $C_0^e$ has cardinality $f$, so $|\Delta(C_0^e)|=f(f-1)$ and $|G^*\setminus{C_0^e}| = q-1-f=f(e-1)$. We can therefore write; 
    \begin{equation*}
        f(f-1) = \lambda(f) + \mu(f(e-1))
    \end{equation*}
and so
    \begin{equation}\label{multi}
        f-1 = \lambda + \mu(e-1).
    \end{equation}
(i) When $e>f$, since $\lambda,\mu$ are non-negative integers, we must have $\mu=0$ and $\lambda=f-1$ in the above equation.  By Proposition \ref{C0_PDS_prop}, the result follows.\\
(ii)If $C_0^e \cup \{0\}$ is a subfield then $|C_0^e|=p^r-1=f$ and $e=\frac{p^{\beta}-1}{p^r-1}$ for some $r \mid \beta$; say $\beta=kr$ ($k>1$). Then  $e=\frac{p^{\beta}-1}{p^r-1}=(p^r)^{k-1}+ (p^r)^{k-2}+ \cdots + (p^r)+1 >p^r-1=f$. The result follows.\\
(iii) Let $e=f>2$. From Equation \ref{multi} above, we have  $f-1 = \lambda + \mu(f-1)$, so either $\mu =0$ or $\mu = 1$.  When $\mu = 0$, then $\lambda = f-1$, and $C_0^e$ is a $(q,f,f-1,0)$-PDS, i.e  a subfield $GF(p^r)$ with 0 removed.  The parameters are $e=\frac{p^{\beta}-1}{p^r-1} > f=p^r-1$, contradicting $e=f$.  So this case cannot occur.  Otherwise $\mu = 1$ and $\lambda = 0$.  The set $C_0^e$ will only form a $(q,f,0,1)$-PDS if the multiset $\Delta(C_0^e)$ comprises of precisely 1 copy of each cyclotomic class $C_{i}^e$, $1 \leq i \leq e-1$ (and no copies of $C_0^e$). By Lemma \ref{lem1CD}, the multiset $\Delta(C_0^e) = \bigcup_{r=1}^{f-1}T_r$, where for each $1 \leq r \leq f-1$, $T_r = C_{a_r}^e$. As $C_0^e$ is a proper PDS, by Corollary \ref{combined_cor} $q \equiv 1 \mod 2e$ and by Lemma \ref{lem4CD} (a)(i) and (b), $-1 \in C_0^e$, hence $T_r=C_{a_r}^e=T_{f-r}$, so for any $r \neq \frac{f}{2}$, there are at least two copies of the cyclotomic class $C_{a_r}^e$ in the multiset $\Delta(C_0^e)$, a contradiction.  We note that for $e=f=2$, i.e. $q=5$, the class $C_0^2$ does form a $(5,2,0,1)$-PDS.
 \end{proof}
 
Note that if $q=ef+1$ is prime and $e>f$, then $C_0^e$ cannot be a proper PDS, as any such PDS must be a subfield of $GF(q)$ with $0$ removed, but $GF(q)$ has no proper subfields when $q$ is prime.

\subsection{Forming DPDFs and EPDFs from cyclotomic PDSs}

In this subsection, we will exhibit constructions of DPDFs and EPDFs by directly taking the sets to be cyclotomic classes, or unions of classes, which themselves form PDSs.

The following is immediate from Proposition \ref{C0_PDS_prop}, Theorem \ref{thm6a} and Theorem \ref{partition_SEDF}.
\begin{proposition}
Let $GF(q)$ be a finite field of order $q=p^{\beta}$ where $\beta>1$.  For $r|\beta$, let $C_0^e$ be the cyclotomic class of order $e$ where $e=\frac{q-1}{p^r-1}$.   Let $\mathcal{S}^{\prime}$ be any collection of $u$ sets ($2 \leq u \leq e-1$) from amongst the cyclotomic classes $\{C_0^e, \ldots, C_{e-1}^e\}$.
\begin{itemize}
    \item[(i)] $\mathcal{S}^{\prime}$ will form a $(q,u,p^r-1, p^r-2,0)$-DPDF.
    \item[(ii)]  If $u=e-1$, $\mathcal{S}^{\prime}$ will also form an $(q,u,p^r-1,q-3p^r+2,q-p^r)$-EPDF.  
    \end{itemize}
\end{proposition}

In the rest of this section, we focus on constructions which simultaneously guarantee DPDFs and EPDFs.

\begin{theorem}
Let $\rm{GF}(q)$ be a finite field of order $q= ef + 1$, where $e \in \{3,4,6,8\}$. Further, let $I \subset \{0,1,\ldots,e-1\}$ (where $|I|=u, 2 \leq u \leq e-1$) and $\mathcal{D}^{\prime} = \{C_i^e\}_{i \in I}$.\\ 
If $C_0^e$ is a PDS (with parameters  $(q,\frac{q-1}{e},\eta^2-(e-3)\eta-1,\eta^2+\eta)$) then $\mathcal{D}^{\prime}$ is both a proper $(q,u,\frac{q-1}{e},u\eta^2+(u+2-e)\eta-1,u(\eta^2+\eta))$-DPDF and a proper  $(q,u,\frac{q-1}{e},u(u-1)\eta^2+2(u-1)\eta,u(u-1)\eta^2)$-EPDF (where $\eta = \frac{(-p)^m-1}{e}$).  
\end{theorem}

\begin{proof}
By Theorem \ref{PDSparams}, for $e \in \{3,4,6,8\}$,  $C_0^e$ is a PDS precisely when $q = p^{2m}$ and $p \equiv -1 \mod e$.  We may therefore apply Theorem \ref{MillsThm}; in its notation, the PDS parameters are $A=(0,0)_e=\eta^2-(e-3)\eta-1$ and $B=(i,0)_e=\eta^2+\eta$ for all $i \neq 0$.  It follows from Theorem 2.7, that $\mathcal{D}^{\prime}$ is both a proper $(q,u,\frac{q-1}{e},u\eta^2+(u+2-e)\eta-1,u(\eta^2+\eta))$-DPDF and a proper  $(q,u,\frac{q-1}{e},u(u-1)\eta^2+2(u-1)\eta,u(u-1)\eta^2)$-EPDF, where $\eta = \frac{s-1}{e}$ for $q = s^2$.  As in the proof of Theorem \ref{ParamThm2}, we can take $s=(-p)^m$.
\end{proof}

\begin{example}\label{ex:25}
Let $q=25$ with $e=6$ and $f=4$; since $q=5^2$ and $5 \equiv -1 \mod 6$, then $C_0^6$ (and hence all $C_i^6$, $0 \leq i \leq 5$) are $(25,4, 3,0)$-PDSs.  Let $\mathcal{D}^{\prime} = \{C_i^6\}_{i \in I}$ where $|I|=u$, $2 \leq u \leq 5$.  If $u=2$ then $\mathcal{D}^{\prime}$ is a $(25,2,4,3,0)$-DPDF/$(25,2,4,0,2)$-EPDF; if $u=3$ then $\mathcal{D}^{\prime}$ is a $(25,3,4,3,0)$-DPDF/$(25,3,4,2,6)$-EPDF; if $u=4$ then $\mathcal{D}^{\prime}$ is a $(25,4,4,3,0)$-DPDF/$(25,4,4,6,12)$-EPDF; and if $u=5$ then $\mathcal{D}^{\prime}$ is a $(25,5,4,3,0)$-DPDF/$(25,5,4,12,20)$-EPDF.  
\end{example}

We have the following recursive construction.  This guarantees we can take each set in our family to be the union of $u$ cyclotomic classes, where the sets are pairwise disjoint, and we obtain a DPDF which is also an EPDF.  The non-trivial examples are always proper, so this does not encompass any of the DDF/EDF constructions in the literature based on taking unions of classes.
\begin{theorem}\label{WDPDF}
Let $\rm{GF}(q^{\prime})$ be a finite field, where $q^{\prime} = q ^{2\beta}$ for some prime power $q$ and $\beta \in \mathbb{N}$. For $e \geq 3$, let $e \mid q+1$ and $\eta = \frac{(-q)^{\beta}-1}{e}$. \\
Let $u,w \in \mathbb{N}$ such that $wu \leq e$ and for $1 \leq a \leq w$, let $I_a \subset \{0,1,\ldots,e-1\}$ such that $|I_a|=u$ and $I_a \cap I_b = \emptyset$ for all $1 \leq a \neq b \leq w$.  Let $D_a = \bigcup\limits_{i \in I_a}C_i^e$, and $\mathcal{W}^{\prime} = \{D_1,D_2,\ldots,D_w\}$.  Then
\begin{itemize}
\item[(i)] $\mathcal{W}^{\prime}$ is a $(q^{\prime},w,u\frac{q^{\prime}-1}{e},u^2\eta^2+(3u-e)\eta-1+(w-1)(u^2\eta^2+u\eta),w(u^2\eta^2+u\eta))$ - DPDF.
\item[(ii)] When $w \geq 2$, $\mathcal{W}^{\prime}$ is a $(q^{\prime},w,u\frac{q^{\prime}-1}{e},w(w-1)u^2\eta^2 + 2(w-1)u\eta,w(w-1)u^2\eta^2)$-EPDF.
\item[(iii)]  If  $\mathcal{W}^{\prime}$ does not partition $GF(q^{\prime})$, then $\mathcal{W}^{\prime}$ is a DDF if and only if each $D_a$ is a difference set.
\item[(iv)] Let $w \geq 2$.  If  $\mathcal{W}^{\prime}$ does not partition $GF(q^{\prime})$, then    $\mathcal{W}^{\prime}$ is not an EDF.
\end{itemize}
\end{theorem}
\begin{proof}
\begin{itemize}
\item[(i)] For each $1 \leq a \leq w$, by Theorem \ref{ParamThm2} we have that $D_a = \bigcup_{i \in I_a}C_i^e$ is a $(q^{\prime},u\frac{q^{\prime}-1}{e},u^2\eta^2+(3u-e)\eta-1,u^2\eta^2+u\eta)$-PDS.  The result then follows from Theorem \ref{thm6a}, after simplifying the parameter expressions.

\item[(ii)] Let $w \geq 2$ and let $W = \cup_{a=1}^{w}D_a$, where  $D_a \in \mathcal{W}^{\prime}$. Since $D_a = \cup_{i \in I_a}C_i^e$, it follows that $W = \cup_{a=1}^{w}\left(\cup_{i \in I_a}C_i^e\right)$, i.e. $W$ is the union of $uw$ $e^{th}$ cyclotomic classes of order $e$. By Theorem \ref{ParamThm2}, $W$ is  $(q^{\prime},uw\frac{q^{\prime}-1}{e},(uw)^2\eta^2 + (3uw-e)\eta-1,(uw)^2\eta^2+uw\eta)$-PDS.

Since $W$ is a PDS that is partitioned by $\mathcal{W}^{\prime}$,  Theorem \ref{thm3} guarantees that $\mathcal{W}^{\prime}$ is also an EDF.  Its parameters may be calculated from the DPDF and PDS parameters according to Theorem \ref{thm3}.
\item[(iii)] If $\mathcal{W}^{\prime}$ does not partition $GF(q^{\prime})$, then it is a DDF when the two frequency parameters are equal.  This happens precisely when $(2u-e)\eta =1$, which by Theorem \ref{ParamThm2} precisely corresponds to the situation when each $D_a$ is a DS. It is immediate that any collection of difference sets forms a DDF.
\item[(iv)]   Let $w \geq 2$. If $\mathcal{W}^{\prime}$ does not partition $GF(q^{\prime})$, then it is an EDF when the two frequency parameters are equal.  This occurs if $(w-1)u\eta = 0$; since $u$ and $\eta$ must be non-zero integers, this implies $w-1=0$ - a contradiction.
\end{itemize}
\end{proof}

\begin{example}
Let $G=GF(49)$, $q=7$ and $e=8$.  In Theorem \ref{WDPDF}  take $u=2$ and $w=3$,  and choose disjoint $2$-sets e.g. $I_1= \{0,2\}$, $I_2 = \{5,6\}$ and $I_3 = \{4,7\}$.  Then $D_1=C_0^8 \cup C_2^8$, $D_2=C_5^8 \cup C_6^8$ and $D_3=C_4^8 \cup C_7^8$ are all (49,12,5,2)-PDSs. Let $\mathcal{W}^{\prime} = \{D_1,D_2,D_3\}$; then $\mathcal{W}^{\prime}$ is a (49,3,12,9,6)-DPDF and a (49,2,12,16,24)-EPDF. 
\end{example}

\section{Partition constructions of DPDFs and EPDFs}

Next, we develop constructions which are simultaneously DPDFs and EPDFs, where the component sets are not necessarily PDSs. Let $q=ef+1=\epsilon \rho+1$ ($\epsilon|e$) be a prime power.

Our approach is to partition $C_0^{\epsilon}$ into smaller cyclotomic classes, when $C_0^{\epsilon}$ is a DS or PDS.  Throughout,  ${C_0^{\epsilon}}^{\prime} = \{C_0^e,C_{\epsilon}^e,\ldots,C_{e-\epsilon}^e \}$ and we will require $e>\epsilon$.
Many known constructions of DDFs and EDFs are in fact of this partition-type (e.g. in \cite{CheLinLin}, \cite{ChaDin} and \cite{HuaWu}), and occur as special cases of the theorems in this section.  All parameters obtainable from the results of this section, with $\epsilon \in \{2,3,4,6,8\}$ and $q \leq 121$, are listed in the second Appendix at the end of the paper.

We first consider the situation when $C_0^{\epsilon}$ is a difference set.  

\begin{theorem}\label{thmg5}
Let \rm{GF}$(q)$ be a finite field of order $q$, where $q=ef+1$ is a prime power and $\epsilon|e$.  Let $C_0^{\epsilon}$ be a $(q,\frac{q-1}{\epsilon},\lambda)$-Difference Set  (here $\lambda=\frac{q-1-\epsilon}{\epsilon^2}$) and denote ${C_0^{\epsilon}}^{\prime} = \{C_0^e,C_{\epsilon}^e,\ldots,C_{e-\epsilon}^e \}$.
\begin{itemize}
\item[(i)] If $\phi_i \neq \phi_j$ for some distinct $i,j \in \{1, \ldots, \epsilon-1\}$ then ${C_0^{\epsilon}}^{\prime}$ is not a DPDF nor an EPDF.
\item[(ii)] Otherwise, ${C_0^{\epsilon}}^{\prime}$ is a $(q,\frac{e}{\epsilon},f,\frac{f-1}{\epsilon})$-DDF and a $(q,\frac{e}{\epsilon},f, {\frac{(e-\epsilon)f}{\epsilon^2})}$-EDF.
\end{itemize}
In particular, ${C_0^{\epsilon}}^{\prime}$ cannot be a proper DPDF nor a proper EPDF.
\end{theorem}

\begin{proof}
By Remark \ref{rem1},
\begin{equation*}
{\rm Int}({C_0^{\epsilon}}^{\prime}) \cup {\rm Ext}({C_0^{\epsilon}}^{\prime}) = \lambda{C_0^{\epsilon}} + \lambda({G}^*\backslash{C_0^{\epsilon}}).    
\end{equation*}
By basic difference set properties, $\lambda(q-1)=(\frac{q-1}{\epsilon})(\frac{q-1}{\epsilon}-1)$, giving the stated value of $\lambda$.

\begin{itemize}
\item[(i)] By Theorem \ref{prop2CD}, $\phi_1=\phi_2=\cdots=\phi_{\epsilon-1}$ is a necessary condition for ${C_0^{\epsilon}}^{\prime}$ to be a DPDF, and hence an EPDF (since $C_0^{\epsilon}$ is a difference set).

\item[(ii)] Since $C_0^{\epsilon}$ is a difference set, by Theorem \ref{DSresults}, $q$ must be odd, $\epsilon$ must be even and $\rho$ must be odd.  By assumption, $\phi_1=\phi_j$ for all $2 \leq j \leq \epsilon-1$.  By Theorem \ref{combined} we have $\phi_0=\phi_{\frac{\epsilon}{2}}$.  Since $\frac{\epsilon}{2} \in \{1, \ldots, \epsilon-1\}$, we have $\phi_0 =\phi_1= \cdots= \phi_{\epsilon-1}$, so by Theorem \ref{prop2CD}, ${C_0^{\epsilon}}^{\prime}$ is a $(q,\frac{e}{\epsilon},f,\phi_0)$-DDF. Since $C_0^{\epsilon}$ is a DS, ${C_0^{\epsilon}}^{\prime}$ is an EDF. 

There are precisely $f-1$ diagonals of transversals in Int$({C_0^{\epsilon}}^{\prime})$, each corresponding to a cyclotomic class $\alpha^iC_0^{\epsilon} = \alpha^i C_0^{\epsilon}$ ($0 \leq i \leq \epsilon-1$), hence $\phi_i=\frac{f-1}{\epsilon}$ for all $0 \leq i \leq \epsilon-1$. So
\begin{equation*}
{\rm Int}({C_0^{\epsilon}}^{\prime}) = \left(\frac{f-1}{\epsilon}\right){G^*} \mbox{ and }
{\rm Ext}({C_0^{\epsilon}}^{\prime})= \left(\lambda-\frac{f-1}{\epsilon} \right){G}^*,
\end{equation*}
and $\lambda-(\frac{f-1}{\epsilon})=\frac{ef-\epsilon}{\epsilon^2}-\frac{f-1}{\epsilon}=\frac{(e-\epsilon)f}{\epsilon^2}$.
\end{itemize}
\end{proof}

The following result, known in the literature (\cite{Ton},\cite{HuaWu}), is an immediate consequence.
\begin{corollary}\label{corg3}
Let GF$(q)$ be a finite field of order $q = ef + 1 \equiv 3 \mod 4$ where $e$ is even.  Denote by $C_0^2$ the set of squares and let ${C_0^2}^{\prime} = \{C_0^e,C_2^e,\ldots,C_{e-2}^e\}$.   Then ${C_0^2}^{\prime}$ is a $(q,\frac{e}{2},f,\frac{f-1}{2})$-DDF and a $(q,\frac{e}{2},f,\frac{(e-2)f}{4})$-EDF. 
\end{corollary}

\begin{proof}
Apply Theorem \ref{thmg5} with $\epsilon=2$; case (i) is impossible as $\epsilon-1=1$, so case (ii) must hold.
\end{proof}

\begin{corollary}\label{corg4}
Let $\rm{GF}(q)$ be a finite field of order $q=ef+1$, where $4|e$, such that $q = 1 + 4t^2$ is the unique proper representation of $q$ and $t$ is odd.  Let ${C_0^4}^{\prime}=\{C_0^e,C_4^e,\ldots,C_{e-4}^e\}$.  If $\phi_1 = \phi_2 = \phi_3$, then ${C_0^4}^{\prime}$ is a $(q,\frac{e}{4},f,\frac{f-1}{4})$-DDF and a $(q,\frac{e}{4},f,\frac{(e-4)f}{16})$-EDF. 
\end{corollary}

\begin{proof}
Apply Theorem \ref{thmg5} with $\epsilon=4$.  By Proposition \ref{DSresults}, $C_0^4$ forms a $(q,\frac{q-1}{4},\frac{q-5}{16})$-Difference Set. Here $\lambda = \frac{q-5}{16}$.
\end{proof}

We note that an EDF result related to the above corollary was given in \cite{CheLinLin}, where it was shown by computational checking that for $q=2917=1+4(27)^2$, $e=324$ and $f=9$ ${C_0^4}^{\prime}$ is a $(2917, 81, 9, 2)$-DDF and a $(2917,81,9,180)$-EDF.

We now consider the situation when we partition a PDS, and show that proper DPDFs and proper EPDFs are obtainable from this (as well as combinations involving DDFs and EDFs).

\begin{theorem}\label{thm6}
Let GF$(q)$ be a finite field of order $q=ef+1$, such that $q$ is a prime power and $\epsilon|e$.  Let $C_0^{\epsilon}$ be a proper $(q,\frac{ef}{\epsilon},\lambda,\mu)$-PDS and denote ${C_0^{\epsilon}}^{\prime} = \{C_0^e,C_{\epsilon}^e,\ldots,C_{e-\epsilon}^e \}$.
\begin{itemize}
\item[(i)] If $\phi_i \neq \phi_j$ for some distinct $i,j \in \{1, \ldots, \epsilon-1\}$ then ${C_0^{\epsilon}}^{\prime}$ is not a DPDF nor an EPDF.  Otherwise, let $\kappa=\phi_1 - \phi_0$; then
\item [(ii)] when $\kappa=0$, ${C_0^{\epsilon}}^{\prime}$ is a $(q,\frac{e}{\epsilon},f,\frac{f-1}{\epsilon})$-DDF and a $(q,\frac{e}{\epsilon},f,\lambda-\frac{f-1}{\epsilon},\mu-\frac{f-1}{\epsilon})$-EPDF.
\item [(iii)] when $\kappa=\mu - \lambda$, ${C_0^{\epsilon}}^{\prime}$ is a  $(q,\frac{e}{\epsilon},f,\frac{(e-\epsilon)f}{\epsilon^2}$)-EDF and a $(q,\frac{e}{\epsilon},f,\lambda-\frac{(e-\epsilon)f}{\epsilon^2},\mu-\frac{(e-\epsilon)f}{\epsilon^2})$-DPDF.
\item [(iv)] when $\kappa \not\in \{0,\mu-\lambda\}$, ${C_0^{\epsilon}}^{\prime}$ is a proper $(q,\frac{e}{\epsilon},f,\phi_0,\phi_0+\kappa)$-DPDF and a proper $(q,\frac{e}{\epsilon},f,\lambda-\phi_0,\mu-\phi_0- \kappa)$-EPDF.
\end{itemize}
\end{theorem}
\begin{proof}
By Remark \ref{rem1},
\begin{equation}
\rm{Int}({C_0^{\epsilon}}^{\prime}) + \rm{Ext}({C_0^{\epsilon}}^{\prime}) = \lambda{C_0^{\epsilon}} + \mu(\rm{G}^*\backslash{C_0^{\epsilon}}).
\end{equation}
$C_0^{\epsilon}$ comprises $\frac{e}{\epsilon}$ cyclotomic classes of the form $C_{\epsilon{i}}^e$, where $0 \leq i \leq \frac{e}{\epsilon}-1$; each cyclotomic class $C_{\epsilon{i}}^e$ has cardinality $f$.

\begin{itemize}
\item[(i)] By Theorem \ref{prop2CD}, $\phi_1=\phi_2= \cdots = \phi_{\epsilon-1}$ is a necessary condition for ${C_0^{\epsilon}}^{\prime}$ to be a DPDF and hence an EPDF (since $C_0^{\epsilon}$ is a PDS).
\item [(ii)] Since $\phi_1=\phi_j$ for all $2 \leq j \leq \epsilon-1$ and $\kappa=0$, i.e $\phi_0 = \phi_1$,  there are $\frac{f-1}{\epsilon}$ transversals corresponding to each cyclotomic class $\alpha^i C_0^{\epsilon}$. Hence, ${C_0^{\epsilon}}^{\prime}$ is a $(q,\frac{e}{\epsilon},f,\frac{f-1}{\epsilon})$-DDF, i.e.
\begin{equation*}
{\rm Int}({C_0^{\epsilon}}^{\prime}) = \left(\frac{f-1}{\epsilon}\right)({G}^*)
\end{equation*}
and so
\begin{equation*}
{\rm Ext}({C_0^{\epsilon}}^{\prime}) = \left(\lambda-\frac{f-1}{\epsilon}\right)(C_0^{\epsilon}) + \left(\mu-\frac{f-1}{\epsilon}\right)({G}\backslash{C_0^{\epsilon}}).
\end{equation*}

\item[(iii)] When $\kappa=\mu-\lambda$, we have $\phi_1 = \phi_0 + \mu - \lambda$. ${C_0^{\epsilon}}^{\prime}$ is a $(q,\frac{ef}{\epsilon},f,\phi_0,\phi_0+\kappa)$-DPDF, with
\begin{equation*}
{\rm Int}({C_0^{\epsilon}}^{\prime}) = \phi_0(C_0^{\epsilon}) + \phi_1({G}^*\backslash{C_0^{\epsilon}}) = \phi_0(C_0^{\epsilon}) + (\phi_0+\kappa)({G}^*\backslash{C_0^{\epsilon}}) = \phi_0(C_0^{\epsilon}) + (\phi_0 + \mu - \lambda)({G}^*\backslash{C_0^{\epsilon}}).
\end{equation*}
We then have
\begin{equation*}
{\rm Ext}({C_0^{\epsilon}}^{\prime}) = (\lambda - \phi_0)C_0^{\epsilon} + (\mu - (\phi_0 + \mu -\lambda))({G}^*\backslash{C_0^{\epsilon}}) = (\lambda - \phi_0){G}^*.   
\end{equation*}

There are $\rho-1 = \frac{ef-\epsilon}{\epsilon}$ transversals in the multiset $\Delta(C_0^{\epsilon})$; each corresponding to some cyclotomic class $\alpha^{i} C_0^{\epsilon}$ ($0 \leq i \leq \epsilon-1$) . There are $f-1$ diagonals of transversals in $\rm{Int}({C_0^{\epsilon}}^{\prime})$, each of which is a transversal of $\Delta(C_0^{\epsilon})$,  so the remaining transversals of the multiset $\Delta(C_0^{\epsilon})$ must be contained within $\rm{Ext}({C_0^{\epsilon}}^{\prime})$. There are
\begin{equation*}
    \frac{ef-\epsilon}{\epsilon} - f-1 = \frac{ef-\epsilon - f\epsilon + \epsilon}{\epsilon} = \frac{(e-\epsilon)f}{\epsilon}
\end{equation*}
of these. As ${C_0^{\epsilon}}^{\prime}$ is an EDF, each of the $\epsilon$ classes occurs equally often, i.e. $\frac{(e-\epsilon)f}{\epsilon^2}$ times.  This implies $\phi_0 = \lambda - \frac{(e-\epsilon)f}{\epsilon^2}$ and $\phi_0 + \kappa = \lambda - \frac{(e-\epsilon)f}{\epsilon^2} + \mu - \lambda = \mu - \frac{(e-\epsilon)f}{\epsilon^2}$.

\item[(iv)] Since $\phi_1=\phi_j$ for $2 \leq j \leq \epsilon$, but $\kappa \neq 0$, ${C_0^{\epsilon}}^{\prime}$ forms a  proper $(q,\frac{e}{\epsilon},f,\phi_0,\phi_0+\kappa)$-DPDF, i.e.
\begin{equation}
{\rm Int}(S') = \phi_0{C_0^{\epsilon}} + (\phi_0+\kappa)({G}^*\backslash{C_0^{\epsilon}}).
\end{equation}
Hence
\begin{equation*}
{\rm Ext}({C_0^{\epsilon}}^{\prime}) = (\lambda-\phi_0){C_0^{\epsilon}} + (\mu-(\phi_0+\kappa))({G}^*\backslash{C_0^{\epsilon}}).
\end{equation*}
As $\kappa \neq \mu - \lambda$, we have $\lambda - \phi_0 \neq \mu - (\phi_0 + \kappa)$, so ${C_0^{\epsilon}}^{\prime}$ is a proper $(q,\frac{e}{\epsilon},f,\lambda - \phi_0,\mu-\phi_0-\kappa)$-EPDF.
\end{itemize}
\end{proof}

The following result is immediate; the EDF case covers the results on partitioning the squares in \cite{CheLinLin}.

\begin{corollary}\label{corg5}
Let $\rm{GF}(q)$ be a finite field of order $q=ef+1 \equiv 1 \mod 4$, where $e$ is even.  Denote by $C_0^2$ the set of squares, let ${C_0^2}^{\prime} = \{C_0^e,C_2^e,...,C_{e-2}^e\}$ and let  $\kappa=\phi_1-\phi_0$.  Then
\begin{itemize}
    \item[(i)] when $\kappa = 0$, ${C_0^2}^{\prime}$ forms a $(q,\frac{e}{2},f,\frac{f-1}{2})$-DDF and a $(q,\frac{e}{2},f,\frac{(e-2)f-2}{4},\frac{(e-2)f+2}{4})$-EPDF. 
    \item[(ii)] when $\kappa = 1$, ${C_0^2}^{\prime}$ forms a $(q,\frac{e}{2},f,\frac{(e-2)f}{4})$-EDF and a $(q,\frac{e}{2},f,\frac{f-2}{2},\frac{f}{2})$-DPDF.
    \item[(iii)] when $\kappa \not\in\{0,1\}$ ${C_0^2}^{\prime}$ is a proper $(q,\frac{e}{2},f,\phi_0,\phi_0+\kappa)$-DPDF and a proper $(q,\frac{e}{2},f,\frac{q-5}{4}-\phi_0,\frac{q-1-4\kappa}{4}-\phi_0)$-EPDF. 
\end{itemize}
\end{corollary}

\begin{proof}
By Proposition \ref{prop1}, ${C_0^2}$ is a $(q,\frac{ef}{2},\frac{q-5}{4},\frac{q-1}{4})$-PDS. Here $\lambda = \frac{q-5}{4}$ and $\mu = \frac{q-1}{4}$. The results then follow from Theorem \ref{thm6} with $\epsilon=2$.
\end{proof}

\begin{example}\label{squares_ex}
Let $\epsilon=2$ and apply Corollary \ref{corg5}.
\begin{itemize}
\item[(i)] Take $q=37$ with $e=4$ and $f=9$.  Then $C_0^4=\{1,7,9,10,12,16,26,33,34\}$, of which $\{10,12,26,34\}$ lie in $\Phi_0$ and $\{7,9,16,33\}$ lie in $\Phi_1$, so $\phi_0=\phi_1=4$, $\kappa=0$ and ${C_0^2}^{\prime}$ is a $(37,2,9,4)$-DDF and a $(37,2,9,4,5)$-EPDF.
\item[(ii)] Take $q=17$ with $e=f=4$; then $\phi_0=1$, $\phi_1=2$, so $\kappa=1$ and ${C_0^2}^{\prime}$ is a $(17,2,4,2)$-EDF and a $(17,2,4,1,2)$-DPDF.
\item[(iii)] Take $q=25$ with $e=6$ and $f=4$; then $\phi_0=3$, $\phi_1=0$, so $\kappa=-3$ and ${C_0^2}^{\prime}$ is a proper $(25,3,4,3,0)$-DPDF and a proper $(25,3,4,2,6)$-EPDF.
\end{itemize}
\end{example}
We observe that EDFs obtained from this process with $f=4$ (eg Example \ref{squares_ex}(ii), (iii)) correspond to those satisfying the parameters of Proposition 21 of \cite{ChaDin}.

Using Proposition \ref{combined} in combination with Corollary \ref{corg5},  we are able to obtain conditions simply in terms of $q$ and $f$ which guarantee that we obtain both a proper DPDF and a proper EPDF.

\begin{theorem}\label{thm8}
Let $q \equiv 1 \mod 4$ where $q=ef+1$ and $e$ even.  Then, if the following conditions hold, ${C_0^2}^{\prime}$ is a proper $(q,\frac{e}{2},f,\phi_0,\phi_1)$-DPDF and a proper $(q,\frac{e}{2},f,\frac{q-5}{4}-\phi_0,\frac{q-1}{4}-\phi_1)$-EPDF:
\begin{itemize}
\item[(i)] $q \equiv 1 \mod 8$ and $f \equiv 2 \mod 4$.
\item[(ii)] $q \equiv 1 \mod 4$, $f \equiv 3 \mod 4$.
\end{itemize}
\end{theorem}

\begin{proof}
\begin{itemize} 
\item[(i)] Since $f$ is even, there are precisely $\frac{f-2}{2}$ values of $1 \leq r < \frac{f}{2}$.  As $f \equiv 2 \mod 4$, this implies $\frac{f-2}{2}$ is even. Since $\psi_0 + \psi_1 = \frac{f-2}{2}$, this implies that $\psi_0$ and $\psi_1$ are either equal or have the same parity. This gives us three cases for the relationship between $\psi_0$ and $\psi_1$; $2\psi_1 = 2\psi_0$, $2\psi_1 -2\psi_0 \geq 4$ or $2\psi_1 -2\psi_0 \leq -4$.\\
We show that in all cases $\kappa \not \in \{0,1\}$ and apply Corollary \ref{corg5}. By Proposition \ref{combined}, when $q \equiv 1 \mod 8$, we have $\phi_0 = 2\psi_0 + 1$ and $\phi_1 = 2\psi_1$.  When $2\psi_1 = 2\psi_0$, we have $\kappa =\phi_1-\phi_0 = -1$.  When $2\psi_1 - 2\psi_0 \geq 4$, this can be rewritten $\phi_1 - (\phi_0-1) \geq 4$, which implies $\kappa \geq 3$.   When $2\psi_1 - 2\psi_0 \leq -4$, this can be rewritten $\phi_1 - \phi_0 + 1 \leq -4$, which implies $\kappa \leq -5$. 
\item[(ii)] This case is similar to the first.
\end{itemize}
\end{proof}

We next present corollaries of Theorem \ref{thm6} for the set of cubes and fourth powers.  The $\epsilon=4$ case for EDFs covers results in \cite{CheLinLin}.
\begin{corollary}\label{corg6}
Let $\rm{GF}(q)$ be a finite field of odd order $q=ef+1 \equiv 1 \mod 3$, where $3|e$. Suppose that  the unique proper representation of $q$, given by $4q = c^2 + 27d^2$ ($c \equiv 1 \mod 3$), has $d=0$.  Let ${C_0^3}^{\prime} = \{C_0^e,C_3^e,\ldots,C_{e-3}^e\}$ and let $\kappa=\phi_1-\phi_0$. If $\phi_1=\phi_2$, then
\begin{itemize}
\item[(i)] when $\kappa=0$, ${C_0^3}^{\prime}$ forms a $(q,\frac{e}{3},f,\frac{f-1}{3})$-DDF and a $(q,\frac{e}{3},f,\frac{(e-3)f-4+c}{9},\frac{2f(e-3)+4-c}{18})$-EPDF.
\item[(ii)] when $\kappa = \frac{4-c}{6}$, ${C_0^3}^{\prime}$ forms a $(q,\frac{e}{3},f,\frac{(e-3)f}{9})$-EDF and a $(q,\frac{e}{3},f,\frac{3f-7+c}{9},\frac{6f-2-c}{18})$-DPDF.
\item[(iii)] when $\kappa \not\in \{0,\frac{4-c}{9}\}$, ${C_0^3}^{\prime}$ forms a proper $(q,\frac{e}{3},f,\phi_0,\phi_0+\kappa)$-DPDF and a proper $(q,\frac{e}{3},f,\frac{q-8+c}{9}-\phi_0,\frac{2q-4-c-18 \kappa}{18}-\phi_0)$-EPDF.
\end{itemize}
\end{corollary}

\begin{proof}
By Theorem \ref{PDSparams}, ${C_0^3}$ is a $(q,\frac{ef}{3},f,\frac{q-8+c}{9},\frac{2q-4-c}{18})$-PDS.  Here $\lambda = \frac{q-8+c}{9}$ and $\mu = \frac{2q-4-c}{18}$. The result follows immediately from Theorem \ref{thm6} with $\epsilon=3$.
\end{proof}

\begin{example}\label{Cubes_25}
\begin{itemize}
\item[(i)]
Let $q=25$ where $e=12$, $f=2$ and $\epsilon=3$.  Since $4q=100=10^2+0.d^2$ with $10 \equiv 1 \mod 3$, we have $c=10$ and ${C_0^3}$ is a $(25,8,3,2)$-PDS.  Here $\phi_0=1$ while $\phi_1=\phi_2=0$, so $\kappa=-1=\frac{4-c}{6}$.  By Corollary \ref{corg6}(ii), ${C_0^3}^{\prime}$ forms a $(25,4,2,1,0)$-DPDF and a $(25,4,2,2)$-EDF.  
\item[(ii)]
Let $q=64$ where $e=9$, $f=7$ and $\epsilon=3$.  Since $4q=256=16^2+0d^2$ with $16 \equiv 1 \mod 3$, we have $c=16$ and ${C_0^3}$ is a $(64,21,8,6)$-PDS.  Then ${C_0^3}^{\prime}$ forms a  $(64,3,7,6,0)$-DPDF and a $(64,3,7,2,6)$-EPDF. 
\end{itemize}
\end{example}

\begin{corollary}\label{corg7}
Let $\rm{GF}(q)$ be a finite field of order $q=p^m=ef+1$ where $4|e$.  Suppose that $q= s^2$ is the proper representation of $q$, where $m$ is even, $p \equiv 3 \mod 4$ and $s = (-p)^{\frac{m}{2}}$. Let ${C_0^4}^{\prime}= \{C_0^e,C_4^e,\ldots,C_{e-4}^e\}$ and let $\kappa = \phi_1-\phi_0$.  If $\phi_1=\phi_2=\phi_3$, then 
\begin{itemize}
    \item[(i)] when $\kappa=0$, ${C_0^4}^{\prime}$ forms a $(q,\frac{e}{4},f,\frac{f-1}{4})$-DDF and a $(q,\frac{e}{4},f,\frac{(e-4)f-6-6s}{16},\frac{(e-4)f+2+2s}{16})$-EPDF.
    \item[(ii)] when $\kappa=\frac{1+s}{2}$, ${C_0^4}^{\prime}$ forms a $(q,\frac{e}{4},f,\frac{(e-4)f}{16})$-EDF and a $(q,\frac{e}{4},f,\frac{4f-10-6s}{16},\frac{4f-2+2s}{16})$-DPDF.
    \item[(iii)] when $\kappa \neq \{0,\frac{1+s}{2}\}$, then ${C_0^4}^{\prime}$ forms a proper $(q,\frac{e}{4},f,\phi_0,\phi_0+\kappa)$-DPDF and a proper $(q,\frac{e}{4},f,\frac{q-11-6s}{16}-\phi_0,\frac{q-3+2s-16 \kappa}{16}-\phi_0)-EPDF$.
\end{itemize}
\end{corollary}

\begin{proof}
 By Theorem \ref{PDSparams}, $C_0^4$ is a $(q,\frac{ef}{4},\frac{q-11-6s}{16},\frac{q-3+2s}{16})$-PDS.
Here $\lambda=\frac{q-11-6s}{16}$ and $\mu=\frac{q-3+2s}{16}$.  The result follows directly from Theorem \ref{thm6} with $\epsilon=4$.
\end{proof}

\begin{example}\label{ex:49}
Take $\epsilon=4$ and let $q=49=7^2$ where $7 \equiv 3 \mod 4$ and $s=-7$ (so $\frac{1+s}{2}=-3$).  ${C_0^4}$ is a $(49,12,5,2)$-PDS.  We apply Corollary \ref{corg7} with various possible values of $e$ divisible by $4$.  Let $\gamma$ be a root of the irreducible polynomial $x^2+4$ over $GF(7)$ and note that $\alpha=\gamma+1$ is a primitive element of $GF(49)$.
\begin{itemize}
\item[(i)] $e=8$ and $f=6$.  Here $C_0^8=\{1, \alpha^8, \alpha^{16}, \alpha^{24}, \alpha^{32}, \alpha^{40} \}=\{1,2,3,4,5,6\} \cong GF(7)^*$.  Here $\phi_0=5$ and $\phi_1=\phi_2=\phi_3=0$, so $\kappa=-5$ and ${C_0^4}^{\prime}$ forms a proper $(49,2,6,5,0)$-DPDF and a proper $(49,2,6,0,2)$-EPDF.  Note this is an example where $C_0^e$ is itself a PDS, formed by removing $0$ from a subfield of $GF(49)$.
\item[(ii)] $e=12$ and $f=4$.  Here $C_0^{12}=\{1, \alpha^{12}, \alpha^{24}, \alpha^{36} \}=\{ 1, 3 \gamma, -1, 4 \gamma\}$.  To satisfy Corollary \ref{corg7}, since $f-1=3$, we would need either $\phi_0=3$ and $\phi_1=\phi_2=\phi_3=0$, or $\phi_0=0$ and $\phi_1=\phi_2=\phi_3=1$.  However $3\gamma-1=\alpha^{10}$ i.e. $3\gamma \in \Phi_2$ whereas $-1-1=-2=\alpha^8$ i.e. $-1 \in \Phi_0$ so neither of these occurs.
\item[(iii)] $e=16$ and $f=3$.  Here $C_0^{16}=\{1, \alpha^{16}, \alpha^{32}\}=\{1,4,2\}$.  Here $4-1=3=\gamma^{40}$ and $2-1=1$ so both elements $2$ and $4$ are in $\Phi_0$, hence $\phi_0=2$ and $\phi_1=\phi_2=\phi_3=0$.  Then $\kappa=-2$ and so ${C_0^4}^{\prime}$ forms a proper $(49,4,3,2,0)$-DPDF and a proper $(49,4,3,3,2)$-EPDF.
\item[(iv)] $e=24$ and $f=2$.  Here $C_0^{24}=\{1,-1\}$ and $-1 \in \Phi_0$ so $\phi_0=1$ and $\phi_1=\phi_2=\phi_3=0$.  Then $\kappa=-1$ and ${C_0^4}^{\prime}$ is a proper $(49,6,2,1,0)$-DPDF and a proper $(49,6,2,4,2)$-EPDF.
\end{itemize}
\end{example}

Next, we present the following general result on the partition construction, based on a similar argument to that of Theorem \ref{eandf} for PDSs. 

\begin{theorem}\label{eandfDPDF}
Let ${C_0^{\epsilon}}^{\prime} = \{C_0^e,C_{\epsilon}^e,\ldots,C_{e-\epsilon}^e \}$ be a $(q, \frac{e}{\epsilon}, f, \phi_0, \phi_1)$-DPDF. 
\begin{itemize}
\item[(i)] If $\epsilon>f$, then $\phi_0=f-1$ and $\phi_1=0$, i.e. ${C_0^{\epsilon}}^{\prime}$ is a $(q,\frac{e}{\epsilon}, f, f-1, 0)$-DPDF.
\item[(ii)] If $\epsilon>2$ and $\epsilon=f$ then ${C_0^{\epsilon}}^{\prime}$ is a $(q,\frac{e}{\epsilon}, f, f-1, 0)$-DPDF.
\item[(iii)] If $C_0^{\epsilon}$ is a $(q,\rho,\lambda,\mu)$-PDS, then in both cases (i) and (ii), ${C_0^{\epsilon}}^{\prime}$ is a $(q, \frac{e}{\epsilon}, f, \lambda-f+1,\mu)$-EPDF, which is proper unless $\mu-\lambda=f-1$.
\end{itemize}
\end{theorem}
\begin{proof}
Similarly to the proof of Theorem \ref{eandf}, the equation $f-1=\phi_0+(\epsilon-1)\phi_1$ holds for ${C_0^{\epsilon}}^{\prime}$.  Note that $f, \epsilon, \phi_0$ and $\phi_1$ are all non-negative integers.  For (i), when $\epsilon>f$, we must have $\phi_1=0$ and hence $\phi_0=f-1$.  For (ii), when $\epsilon=f$, there are two cases: $\phi_1=0$ and $\phi_0=f-1$ or $\phi_1=1$ and $\phi_0=0$.  By Theorem \ref{combined}, when $q \equiv 1 \mod 2 \epsilon$ and $f$ even ($\epsilon>2$) or $f$ odd, and when $q \equiv \epsilon+1 \mod 2 \epsilon$ and $\epsilon$ odd, we have $\phi_1$ even and so the second case cannot occur.  We cannot have $q \equiv \epsilon+1 \mod 2 \epsilon$ (i.e. $\rho$ odd) and $\epsilon$ even, as in this case $f$ is odd so cannot equal $\epsilon$.  Hence only the first case is possible, unless $\epsilon=f=2$.  Part (iii) is immediate from Theorem \ref{thm6}.
\end{proof}
Illustrations of  Theorem \ref{eandfDPDF} (i) may be found in Examples \ref{Cubes_25} and \ref{ex:49}.

Finally, in the situation when $\epsilon=f=2$, the following explicit result can be obtained.   This gives examples of DPDFs corresponding to both possible cases in the proof of Theorem \ref{eandfDPDF}(ii).  The EDF result for  $q \equiv 5 \mod 8$ covers the corresponding case in Lemma 19 of \cite{ChaDin}.

\begin{theorem}\label{thm7}
Let $q=ef+1 \equiv 1 \mod 4$ with $e$ even and $f=2$.  Then;
\begin{itemize}
\item[(i)] when $q \equiv 1 \mod 8$, ${C_0^2}^{\prime}$ forms a $(q,\frac{q-1}{4},2,1,0)$-DPDF and a $(q,\frac{q-1}{4},2,\frac{q-9}{4},\frac{q-1}{4})$-EPDF.
\item[(ii)] when $q \equiv 5 \mod 8$, ${C_0^2}^{\prime}$ forms a $(q,\frac{q-1}{4},2,0,1)$-DPDF and a $(q,\frac{q-1}{4},2,\frac{q-5}{4})$-EDF.
\end{itemize}
\end{theorem}

\begin{proof}
By Proposition \ref{prop1}, when $q \equiv 1 \mod 4$, $C_0^2$ forms a $(q,\frac{q-2}{2},\frac{q-5}{4},\frac{q-1}{4})$-PDS. By Remark \ref{rem1}
\begin{equation}
{\rm Int}({C_0^2}^{\prime}) + {\rm Ext}({C_0^2}^{\prime}) = \frac{q-5}{4}(C_0^2) + \frac{q-1}{4}(C_1^2).   
\end{equation}
When $f=2$, ${\rm Int}({C_0^2}^{\prime}) = D_1 = (\alpha^e-1)S$, i.e. ${\rm Int}({C_0^2}^{\prime}) = D_{\frac{f}{2}}$.

For part (i), by Lemma \ref{central}, when $q \equiv 1 \mod 8$ we have $D_{\frac{f}{2}}=C_0^2$, hence $\rm{Int}({C_0^2}^{\prime}) = C_0^2$. So ${C_0^2}^{\prime}$ forms a $(q,\frac{q-1}{4},2,1,0)$-DPDF. Moreover, 
${\rm Ext}({C_0^2}^{\prime}) = \left(\frac{q-9}{4}\right)(C_0^2) + \left(\frac{q-1}{4}\right)(C_1^2)$ so that ${C_0^2}^{\prime}$ is also an $(q,\frac{q-1}{4},2,\frac{q-9}{4},\frac{q-1}{4})$-EPDF.  Part (ii) can be proved similarly.
\end{proof}

\begin{example}\label{eg3}
\begin{itemize}
\item[(i)]
In \rm{GF}$(25)$, if $f=2$, then $e=\frac{q-1}{2}=12$. Then $C_0^e = \langle\alpha^{12}\rangle$ and ${C_0^2}^{\prime}$ is a $(25,6,2,1,0)$-DPDF and a $(25,6,2,4,6)$-EPDF, since $25 \equiv 1 \mod 8$.
\item[(ii)]
In
 \rm{GF}$(13)$, if $f=2$, then this implies that $e=\frac{q-1}{2}=6$. Then $C_0^e = \langle\alpha^{6}\rangle$ and ${C_0^2}^{\prime}$ is a $(13,3,2,0,1)$-DPDF and a $(13,3,2,2)$-EDF, since $13 \equiv 5 \mod 8$.
\end{itemize}
\end{example}

\subsection{Explicit results for cyclotomic DPDFs and EDPFs which partition the squares}

In this section, we obtain explicit information about the DPDFs and EPDFs constructed using $\epsilon=2$ in Theorem \ref{thm6} using cyclotomic numbers for certain values of even $e$. Necessary cyclotomic results are given in the Appendix.

The EDF part of the following result corresponds to that of Theorem 3.2 of \cite{HuaWu}.
\begin{theorem}
Let $q$ be an odd prime power of the form $q=4f+1$.  Let $s$ be defined as it is in Theorem 7.2. Then,
\begin{itemize}
\item[(i)] Suppose $f$ is even. Then ${C_0^2}^{\prime}$ is a $(q, 2,\frac{q-1}{4}, \frac{q-7-2s}{8}, \frac{q-3+2s}{8})$-DPDF and a $(q,2,\frac{q-1}{4}, \frac{q-3+2s}{8}, \frac{q+1-2s}{8})$-EPDF which are both proper except for when $s=1$, in which case ${C_0^2}^{\prime}$ is a $(q,2,\frac{q-1}{4}, \frac{q-1}{8})$-EDF.

\item[(ii)]  Suppose $f$ is odd. Then ${C_0^2}^{\prime}$ is a $(q,2,\frac{q-1}{4}, \frac{q-7+2s}{8}, \frac{q-3-2s}{8})$-DPDF and a $(q,2,\frac{q-1}{4}, \frac{q-3-2s}{8}, \frac{q+1+2s}{8})$-EPDF which are both proper except when $s=1$, in which case ${C_0^2}^{\prime}$ is a $(q,2,\frac{q-1}{4}, \frac{q-5}{8})$-DDF.
\end{itemize}
\end{theorem}

\begin{proof}
In the notation of Theorem \ref{KaRaThm},
\begin{itemize} 
\item[(i)] if $f$ is even, $\phi_0 = A+C = \frac{q-2s-7}{8}$ and $\phi_1 = B+D = \frac{q+2s-3}{8}$;  
\item [(ii)] if $f$ is odd, $\phi_0 = 2A = \frac{q+2s-7}{8}$ and $\phi_1 = 2E = \frac{q-2s-3}{8}$.
\end{itemize}
The result follows by combining this with Theorem \ref{thm6} and Corollary \ref{combined_cor}.
\end{proof}

We can repeat this process for $e=6$. Before we look at the cyclotomic numbers for $e=6$, we will make the following general observation about the $f$ odd case. 

\begin{remark}
Let GF$(q)$ be a finite field of prime power order such that $q = 6f+1$ such that $f$ is odd. Then $q \equiv 3 \mod 4$ and by Theorem \ref{thmg5}, this implies that ${C_0^2}^{\prime}$ is always a $(q,3,f,\frac{f-1}{2})$-DDF and a $(q,3,f,f)$-EDF. 
\end{remark}

With the case for $f$ odd ruled out when $e=6$, this just leaves the case for $f$ even.  The EDF part corresponds to the EDF result in Corollary 2.10 of \cite{HuaWu}.

\begin{theorem}\label{Cyce=6}
Let $q$ be an odd prime power of the form $q = 6f+1$ and suppose $f$ is even. Let $s$ be defined as it is in Theorem 7.3. Then ${C_0^2}^{\prime}$ is a proper $(q,3,\frac{q-1}{6},\frac{q-9-4s}{12},\frac{q-5+4s}{12})$-DPDF and a proper $(q,3,\frac{q-1}{6},\frac{q-3+2s}{6},\frac{q+1-2s}{6})$-EPDF except when $s = 1$, in which case ${C_0^2}^{\prime}$ is a proper $(q,3,\frac{q-1}{6},\frac{q-13}{12},\frac{q-1}{12})$-DPDF and a $(q,3,\frac{q-1}{6},\frac{q-1}{6})$-EDF.
\end{theorem}
\begin{proof}
In the notation of Theorem \ref{thm:e=6}: when $f$ is even, $\phi_0 = A + C + E = \frac{q-9-4s}{12}$ and $\phi_1 = B + D + F = \frac{q-5+4s}{12}$ (since the values of $\phi_0$ and $\phi_1$ are the same in each of the cases stated in Theorem \ref{thm:e=6}).  The result follows by combining this with Theorem \ref{thm6} and Corollary \ref{combined_cor}.
\end{proof}

\begin{example}
The $GF(13)$ construction shown in Example \ref{eg3}, is also an illustration of Theorem \ref{Cyce=6}.  Here, $q=13$, $e=6$ and $f=2$; since $13=1+3 \cdot 2^2$, $s=1$ and so ${C_0^2}^{\prime}$ is a proper $(13,3,2,0,1)$-DPDF and a $(13,6,2,2)$-EDF.
\end{example}

We can also replicate this process for $e=8$. 

\begin{theorem}
Let $q$ be an odd prime power of the form $q = 8f + 1$. Let $x$ and $a$ be defined as they are in Theorem 7.4. Then ${C_0^2}^{\prime}$ is always a proper $(q,4,\frac{q-1}{8},\frac{q-11-2x-4a}{16},\frac{q-7+2x+4a}{16})$-DPDF and a proper $(q,4,\frac{q-1}{8},\frac{3q-9+2x+4a}{16},\frac{3q+3-2x-4a}{16})$-EPDF, except in the following cases:
\begin{itemize}
\item[(i)] When $f$ is even and $x+2a=3$. In this case ${C_0^2}^{\prime}$ is a $(q,4,\frac{q-1}{8},\frac{3q-3}{16})$-EDF. 
\item[(ii)] When $f$ is odd and $x+2a=-1$. In this case ${C_0^2}^{\prime}$ is a $(q,4,f,\frac{q-9}{16})$-DDF.    
\end{itemize}
\end{theorem}
\begin{proof}
Suppose $q$ is represented by $q=x^2+4y^2=a^2+2b^2$ as in Theorem \ref{thm:e=8}. Then 
\begin{itemize}
\item[(i)] When $f$ is even, $\phi_0 = A + C + E + G = \frac{q-11-2x-4a}{16}$ and $\phi_1 = B + D + F +H= \frac{q-7+2x+4a}{16}$.
\item[(ii)] When $f$ is odd, $\phi_0 = 2A + 2N = \frac{q-11-2x-4a}{16}$ and $\phi_1 = 2I + 2J = \frac{q-7+2x+4a}{16}$.
\end{itemize}
(Note that the same values of $\phi_0$ and $\phi_1$ hold irrespective of whether 2 is a quartic residue). The result follows by combining this with Theorem \ref{thm6} and  Corollary \ref{combined_cor}.
\end{proof}

\begin{example}
Let GF$(q)$ = GF$(41)$, such that $e = 8$ and $f = 5$. Here $q = 5^2 + 4(2^2) = (-3)^2 + 2(4^2)$, and as $5,-3 \equiv 1 \mod 4$, we can see that these are both proper representations of $q$. Notice that $x + 2a = 5 - 6 = -1$. Since $f$ is odd and $x + 2a = -1$,  ${C_0^2}^{\prime}$ forms a $(41,4,5,2)$-DDF and a proper $(41,4,5,7,8)$-EPDF.
\end{example}

\section{Further Work}
It would be of interest to find new constructions for DPDFs and EPDFs.  For the cyclotomic approach, one natural direction would be to further investigate theoretical constructions where the sets in the family are formed from unions of $C_i^e$.  An initial investigation was begun in this paper via Theorem \ref{WDPDF}, and there are various examples of such DDF/EDF constructions in the literature which are not covered by this theorem.

We may also ask about non-cyclotomic construction methods.   Packings of PDSs offer a method of obtaining DPDFs, and when they have the additional property that the complement of the union is a subgroup then we are guaranteed an EPDF.   For example, in the recent paper \cite{JedLi}, the authors define a $(c,t)$ LP-packing $\{P_1, \ldots, P_t\}$ in an abelian group $G$ of order $t^2 c^2$ relative to a subgroup $U$ of order $tc$, and prove an existence result.  Here the $P_i$'s are disjoint, each of order $c(tc-1)$, and the complement of their union is a subgroup $U$ of $G$.  Any such $(c,t)$ LP-packing will form a DPDF which is also an EPDF.  It would be of interest to explore such approaches further.

There are various ways in which the definitions of DPDF and EPDF could naturally be generalized.  The condition on disjointness of the sets in the DPDF could be relaxed, and the set sizes could be allowed to vary.  We could also generalize our DPDF/EPDF definitions in such a way that the two frequencies correspond to elements inside/outside an arbitrary subset of $G$, rather than the union of the sets of the family. 

\begin{definition}\label{T_PDF}
Let $T$ be a subset of $G^*$. 
\begin{itemize}
\item[(i)] A collection of $m$ disjoint subsets $S = \{ D_1, \ldots, D_m \}$ in $G^*$, where each $D_i$ has cardinality $k_i$, forms an $(n,m, k_1, \ldots, k_m; \lambda, \mu)$ Disjoint Partial Difference Family of $G$ (relative to $T$) if the following multiset equation holds: 
$$ \cup_i \Delta(D_i)= \lambda(T) + \mu (G^* \setminus T) $$
\item[(ii)] A collection of disjoint $m$ subsets $S = \{ D_1, \ldots, D_m \}$ in $G^*$, where each $D_i$ has cardinality $k_i$, forms an $(n,m, k_1, \ldots, k_m; \lambda, \mu)$ External Partial Difference Family (relative to $T$) if the following multiset equation holds: 
$$ \cup_{i \neq j} \Delta(D_i, D_j)= \lambda(T) + \mu (G^* \setminus T) $$
\end{itemize}
\end{definition}
Observe that Definitions \ref{PDF} and \ref{EPDF} correspond to taking $T=\cup_{i=1}^m D_i$ in Definition \ref{T_PDF}.  This is also closer to the definition given in \cite{DavHucMul}.  Construction methods for these would be of interest.

\section{Appendix: cyclotomic numbers}
In this appendix, we briefly present the necessary literature on evaluating cyclotomic numbers for $e=3,4,6$ and $8$.  Our references are \cite{Dic}, \cite{KatRaj}, \cite{Leh},\cite{Sto} and the Appendix in \cite{MeiWin}.

\begin{theorem}\label{thm:e=3}
Let GF$(q)$ be a finite field of order $q=p^m \equiv 1 \mod 3$.  Let 
$$4q = c^2 + 27d^2, \qquad c \equiv 1 \mod 3$$ 
where \begin{itemize}
    \item if $p \equiv 2 \mod 3$ then $m$ is even and $d=0$;
    \item if $p \equiv 1 \mod 3$ then this is the unique proper representation of $4q$, with $c \equiv 1 \mod 3$. 
\end{itemize}
We have the following cyclotomic relations; 
\begin{center} 
$A = (0,0) = \frac{1}{9}(q-8+c)$, $B = (0,1) = (1,0) = (2,2) = \frac{1}{18}(2q-4-c-9d)$, $C = (0,2) = (1,1) = (2,0) = \frac{1}{18}(2q-4-c+9d)$, \\
$D = (1,2) = (2,1) = \frac{1}{9}(q+1+c)$.
\end{center}
\end{theorem}

\begin{theorem}\label{KaRaThm}
Let $p$ be an odd prime, let $q=p^m \equiv 1 \mod 4$, and write $q=4f+1$.  Let $v$ be a generator of $GF(q)^*$. 
\begin{itemize}
    \item If $p \equiv 3 \mod 4$, let $s=(-p)^{m/2}$ and $t=0$.
    \item If $p \equiv 1 \mod 4$, define $s$ uniquely by $q=s^2+t^2$, $p \nmid s$, $s \equiv 1 \mod 4$, then $t$ uniquely by $v^{(q-1)/4} \equiv s/t \mod p$.
\end{itemize}
Then the cyclotomic numbers $(i,0)$ of order $4$ for $GF(q)$, corresponding to $v$, are determined unambiguously by the following formulae:
\begin{itemize}
    \item For $f$ even: $A=(0,0)=\frac{1}{16}(q-11-6s)$; $B=(1,0)=(0,1)=(3,3)=\frac{1}{16}(q-3+2s+4t)$; $C=(2,0)=(0,2)=(2,2)=\frac{1}{16}(q-3+2s)$; $D=(3,0)=(0,3)=(1,1)=\frac{1}{16}(q-3+2s-4t)$; 
    $E=(1,2)=(1,3)=(2,1)=(2,3)=(3,1)=(3,2)=\frac{1}{16}(q+1-2s)$.
    \item For $f$ odd: $A=(0,0)=(2,0)=\frac{1}{16}(q-7+2s)$; $E=(1,0)=(3,0)=\frac{1}{16}(q-3-2s)$.
\end{itemize}
\end{theorem}

\begin{theorem}\label{thm:e=6}
Let $q=6f+1$, ie $e=6$.  When $e=6$, the $36$ cyclotomic numbers are solely functions of the unique representation 
$$ q=p^m=s^2+3t^2; \qquad s \equiv 1 \mod 3 $$
determined by
\begin{itemize}
    \item if $p \equiv 5 \mod 6$ then $m$ is even and $q=(\pm p^{m/2})^2+ 3 (0)^2$;
    \item if $p \equiv 1 \mod 6$ then $q=s^2+3t^2$ is the unique proper representation of $q$, with $s \equiv 1 \mod 3$; the sign of $t$ is ambiguously determined.
\end{itemize}
Then, when $f$ is even, the cyclotomic numbers $(i,0)$ of order 6 for $\rm{GF}(q)$ may be determined by the following formulae: 
\begin{itemize}
\item when $2 \in C_0$ or $2 \in C_3$: 
$A = (0,0) = \frac{1}{36}(q-17-20s)$;
$B = (1,0) = \frac{1}{36}(q-5+4s+18t)$;
$C = (2,0) = \frac{1}{36}(q-5+4s+6t)$;
$D = (3,0) = \frac{1}{36}(q-5+4s)$;
$E = (4,0) = \frac{1}{36}(q-5+4s-6t)$;
$F = (5,0) = \frac{1}{36}(q-5+4s-18t)$.
\item when $2 \in C_1$ or $2 \in C_4$:
$A = (0,0) = \frac{1}{36}(q-17-8s+6t)$;
$B = (1,0) = \frac{1}{36}(q-5+4s+12t)$;
$C = (2,0) = \frac{1}{36}(q-5+4s-6t)$;
$D = (3,0) = \frac{1}{36}(q-5+4s-6t)$
$E = (4,0) = \frac{1}{36}(q-5-8s)$.
$F = (5,0) = \frac{1}{36}(q-5+4s-6t)$
\item when $2 \in C_2$ or $2 \in C_5$:
$A = (0,0) = \frac{1}{36}(q-17-8s-6t)$;
$B = (1,0) = \frac{1}{36}(q-5+4s+6t)$;
$C = (2,0) = \frac{1}{36}(q-5-8s)$;
$D = (3,0) = \frac{1}{36}(q-5+4s+6t)$
$E = (4,0) = \frac{1}{36}(q-5+4s+6t)$
$F = (5,0) = \frac{1}{36}(q-5+4s-12t)$.
\end{itemize}
\end{theorem}

\begin{theorem}\label{thm:e=8}
Let $q=8f+1$.  When $e=8$, the cyclotomic numbers are uniquely determined by $x,y,a$ and $b$, defined as follows.
\begin{itemize}
    \item $q=x^2+4y^2$, $x \equiv 1 \mod 4$ is the unique proper representation of $q=p^m$ if $p \equiv 1 \mod 4$; otherwise $x=\pm p^{m/2}$ and $y=0$.
    \item $q=a^2+2b^2$, $a \equiv 1 \mod 4$, is the unique proper representation of $q=p^m$ if $p \equiv 1$ or $3 \mod 8$; otherwise $a= \pm p^{m/2}$ and $b=0$. 
\end{itemize}
The signs of $y$ and $b$ are ambiguously determined.\\
The cyclotomic numbers $(i,0)$ of order 8 for $\rm{GF}(q)$ are determined by the following formulae: 
\begin{itemize}
\item When 2 is a quartic residue and $f$ is even: 
$A = (0,0) = \frac{1}{64}(q-23-18x-24a)$;
$B = (1,0) = \frac{1}{64}(q-7+2x+4a+16y+16b)$;
$C = (2,0) = \frac{1}{64}(q-7+6x+16y)$;
$D = (3,0) = \frac{1}{64}(q-7+2x+4a-16y+16b)$;
$E = (4,0) = \frac{1}{64}(q-7-2x+8a)$;
$F = (5,0) = \frac{1}{64}(q-7+2x+4a+16y-16b)$;
$G = (6,0) = \frac{1}{64}(q-7+6x-16y)$;
$H = (7,0) = \frac{1}{64}(q-7+2x+4a-16y-16b)$.
\item When 2 is a quartic residue and $f$ is odd: 
$A = (0,0) = (4,0) = \frac{1}{64}(q-15-2x)$;
$I = (1,0) = (5,0) = \frac{1}{64}(q-7+2x+4a)$;
$N = (2,0) = (6,0) = \frac{1}{64}(q-7-2x-8a)$;
$J = (3,0) = (5,0) = \frac{1}{64}(q-7+2x+4a)$.
\item When 2 is not a quartic residue and $f$ is even: 
$A = (0,0) = \frac{1}{64}(q-23+6x)$;
$B = (1,0) = \frac{1}{64}(q-7+2x+4a)$;
$C = (2,0) = \frac{1}{64}(q-7-2x-8a-16y)$;
$D = (3,0) = \frac{1}{64}(q-7+2x+4a)$;
$E = (4,0) = \frac{1}{64}(q-7-10x)$;
$F = (5,0) = \frac{1}{64}(q-7+2x+4a)$;
$G = (6,0) = \frac{1}{64}(q-7-2x-8a+16y)$;
$H = (7,0) = \frac{1}{64}(q-7+2x+4a)$
\item When 2 is not a quartic residue and $f$ is odd:
$A = (0,0) = (4,0) = \frac{1}{64}(q-15-10x-8a)$;
$I = (1,0) = (5,0) = \frac{1}{64}(q-7+2x+4a+16y)$;
$N = (2,0) = (6,0) = \frac{1}{64}(q-7+6x)$;
$J = (3,0) = (7,0) = \frac{1}{64}(q-7+2x+4a-16y)$.
\end{itemize}
\end{theorem}

\section{Appendix: Section 5 DPDF/EPDF parameters}
The following table contains all Section 5 parameters for $q \leq 121$ and $e \in \{2,3,4,6,8\}$.

\begin{center}
{\scriptsize
\begin{longtable}[c]{|c|c|c|c|c|c|c|c|}
    \hline
    $q$ & $\epsilon$ & $\rho$ & PDS Parameters & $e$ & $f$ & DPDF Parameters & EPDF Parameters  \\
    \hline 
    9 & 2 & 4 & (9,4,1,2)-PDS & 4 & 2 & (9,2,2,1,0)-DPDF & (9,2,2,0,2)-EPDF \\
    \hline
    13 & 2 & 6 & (13,6,2,3)-PDS & 4 & 3 & (13,2,3,0,2)-DPDF & (13,2,3,2,1)-EPDF \\
    \hline
    13 & 2 & 6 & (13,6,2,3)-PDS & 6 & 2 & (13,2,3,0,1)-DPDF & (13,2,3,2)-EDF \\
    \hline
    17 & 2 & 8 & (17,8,3,4)-PDS & 4 & 4 & (17,2,4,1,2)-DPDF & (17,2,4,2)-(S)EDF \\
    \hline
    17 & 2 & 8 & (17,8,3,4)-PDS & 8 & 2 & (17,4,2,1,0)-DPDF & (17,4,2,2,4)-EPDF \\
    \hline
    25 & 2 & 12 & (25,12,5,6)-PDS & 4 & 6 & (25,2,6,3,2)-DPDF & (25,2,6,2,4)-EPDF \\
\hline
    25 & 2 & 12 & (25,12,5,6)-PDS & 6 & 4 & (25,3,4,3,0)-DPDF & (25,3,4,2,6)-EPDF \\
    \hline
    25 & 2 & 12 & (25,12,5,6)-PDS & 8 & 3 & (25,4,3,0,2)-DPDF & (25,4,3,5,4)-EPDF \\
    \hline
    25 & 2 & 12 & (25,12,5,6)-PDS & 12 & 2 & (25,6,2,1,0)-DPDF & (25,6,2,4,6)-EPDF \\
    \hline
     25 & 3 & 8 & (25,8,3,2)-PDS & 6 & 4 & (25,2,4,3,0)-DPDF & (25,2,4,0,2)-EPDF \\
    \hline
    25 & 3 & 8 & (25,8,3,2)-PDS & 12 & 2 & (25,4,2,1,0)-DPDF & (25,4,2,2)-EDF \\
    \hline
     25 & 6 & 4 & (25,4,3,0)-PDS & 12 & 2 & (25,2,2,1,0)-DPDF & (25,2,2,2,0)-EPDF \\
    \hline
    29 & 2 & 14 & (29,14,6,7)-PDS & 4 & 7 & (29,2,7,4,2)-DPDF & (29,2,7,2,5)-EPDF \\
    \hline
    29 & 2 & 14 & (29,14,6,7)-PDS & 14 & 2 & (29,7,2,0,1)-DPDF & (29,7,2,6)-EDF \\
    \hline
    37 & 2 & 18 & (37,18,8,9)-PDS & 4 & 9 & (37,2,9,4)-DDF & (29,2,9,4,5)-EPDF \\
    \hline
    37 & 2 & 18 & (37,18,8,9)-PDS & 6 & 6 & (37,3,6,4,1)-DPDF & (37,3,6,4,8)-EPDF \\
    \hline
    37 & 2 & 18 & (37,18,8,9)-PDS & 12 & 3 & (37,6,3,2,0)-DPDF & (37,6,3,6,9)-EPDF \\
    \hline
    37 & 2 & 18 & (37,18,8,9)-PDS & 18 & 2 & (37,9,2,0,1)-DPDF & (37,9,2,8)-EDF \\
    \hline
    41 & 2 & 20 & (41,20,9,10)-PDS & 4 & 10 & (41,2,10,3,6)-DPDF & (41,2,10,6,4)-EPDF \\
    \hline
    41 & 2 & 20 & (41,20,9,10)-PDS & 8 & 5 & (41,4,5,2)-DDF & (41,4,5,7,8)-EPDF \\
    \hline
    41 & 2 & 20 & (41,20,9,10)-PDS & 10 & 4 & (41,5,4,3,0)-DPDF & (41,5,4,6,10)-EPDF \\
    \hline
    41 & 2 & 20 & (41,20,9,10)-PDS & 20 & 2 & (41,10,2,1,0)-DPDF & (41,10,2,8,10)-EPDF \\
    \hline
    49 & 2 & 24 & (49,24,11,12)-PDS & 4 & 12 & (49,2,12,5,6)-DPDF & (49,2,12,6)-EDF \\
    \hline
    49 & 2 & 24 & (49,24,11,12)-PDS & 6 & 8 & (49,3,8,3,4)-DPDF & (49,3,8,8)-EDF \\
    \hline
    49 & 2 & 24 & (49,24,11,12)-PDS & 8 & 6 & (49,4,6,5,0)-DPDF & (49,4,6,6,12)-EDF \\
    \hline
    49 & 2 & 24 & (49,24,11,12)-PDS & 12 & 4 & (49,6,4,3,0)-DPDF & (49,2,8,12)-EPDF \\
    \hline
    49 & 2 & 24 & (49,24,11,12)-PDS & 16 & 3 & (49,8,3,2,0)-DPDF & (49,8,3,9,12)-EPDF \\
    \hline
    49 & 2 & 24 & (49,24,11,12)-PDS & 24 & 2 & (49,12,2,1,0)-DPDF & (49,12,2,10,12)-EPDF \\
    \hline
     49 & 4 & 12 & (49,12,5,2)-PDS & 8 & 6 & (49,2,6,5,0)-DPDF &(49,2,6,0,2)-EPDF \\
    \hline
    49 & 4 & 12 & (49,12,5,2)-PDS & 16 & 3 & (49,4,3,2,0)-DPDF & (49,4,3,3,2)-EPDF \\
    \hline
    49 & 4 & 12 & (49,12,5,2)-PDS & 24 & 2 & (49,6,2,1,0)-DPDF & (49,6,2,4,2)-EPDF \\
    \hline
     49 & 8 & 6 & (49,8,5,0)-PDS & 16 & 3 &  (49,2,3,2,0)-DPDF & (49,2,3,3,0)-EPDF \\
    \hline
    49 & 8 & 6 & (49,8,5,0)-PDS & 24 & 2 &  (49,3,2,1,0)-DPDF & (49,3,2,4,0)-EPDF \\
    \hline
    53 & 2 & 26 & (53,26,12,13)-PDS & 4 & 13 & (53,2,13,4,8)-DPDF & (53,2,13,8,5)-EPDF \\
    \hline
    53 & 2 & 26 & (53,26,12,13)-PDS & 26 & 2 & (53,13,2,0,1)-DPDF & (53,2,13,12)-EDF \\
    \hline
    61 & 2 & 30 & (61,30,14,15)-PDS & 4 & 15 & (61,2,15,8,6)-DPDF & (61,2,15,6,9)-EPDF \\
    \hline
    61 & 2 & 30 & (61,30,14,15)-PDS & 6 & 10 & (61,3,10,2,7)-DPDF & (61,3,10,12,8)-EPDF \\
    \hline
    61 & 2 & 30 & (61,30,14,15)-PDS & 10 & 6 & (61,5,6,4,1)-DPDF & (61,5,6,10,14)-EPDF \\
    \hline
    61 & 2 & 30 & (61,30,14,15)-PDS & 12 & 5 & (61,6,5,2)-DDF & (61,6,5,12,13)-EPDF \\
    \hline
    61 & 2 & 30 & (61,30,14,15)-PDS & 20 & 3 & (61,10,3,2,0)-DPDF & (61,6,5,12,15)-EPDF \\
    \hline
    61 & 2 & 30 & (61,30,14,15)-PDS & 30 & 2 & (61,15,2,0,1)-DPDF & (61,15,2,14)-EDF \\
    \hline
     64 & 3 & 21 & (64,21,8,6)-PDS & 9 & 7 & (64,3,7,6,0)-DPDF & (64,3,7,2,6)-EPDF \\
    \hline
    64 & 3 & 21 & (64,21,8,6)-PDS & 21 & 3 & (64,7,3,2,0)-DPDF & (64,7,3,6)-EDF \\
    \hline
    73 & 2 & 36 & (73,36,17,18)-PDS & 4 & 18 & (73,2,18,9,8)-DPDF & (73,2,18,8,10)-EPDF \\
    \hline
    73 & 2 & 36 & (73,36,17,18)-PDS & 6 & 12 & (73,3,18,7,4)-DPDF & (73,3,12,10,14)-EPDF \\
    \hline
    73 & 2 & 36 & (73,36,17,18)-PDS & 8 & 9 & (73,4,9,4)-DDF & (73,4,9,13,14)-EPDF \\
    \hline
    73 & 2 & 36 & (73,36,17,18)-PDS & 12 & 6 & (73,6,6,3,2)-DPDF & (73,6,6,14,16)-EPDF \\
    \hline
    73 & 2 & 36 & (73,36,17,18)-PDS & 18 & 4 & (73,9,4,1,2)-DPDF & (73,9,4,16)-EDF \\
    \hline
    73 & 2 & 36 & (73,36,17,18)-PDS & 24 & 3 & (73,12,3,0,2)-DPDF & (73,12,3,17,16)-EPDF \\
    \hline
    73 & 2 & 36 & (73,36,17,18)-PDS & 36 & 2 & (73,18,2,1,0)-DPDF & (73,18,2,16,18)-EPDF \\
    \hline
    81 & 2 & 40 & (81,40,19,20)-PDS & 4 & 20 & (81,2,20,7,12)-DPDF & (81,2,20,12,8)-EPDF \\
    \hline
    81 & 2 & 40 & (81,40,19,20)-PDS & 8 & 10 & (81,4,10,5,4)-DPDF & (81,4,10,14,16)-EPDF \\
    \hline
    81 & 2 & 40 & (81,40,19,20)-PDS & 10 & 8 & (81,5,8,7,0)-DPDF & (81,5,8,12,20)-EPDF \\
    \hline
    81 & 2 & 40 & (81,40,19,20)-PDS & 16 & 5 & (81,8,5,0,4)-DPDF & (81,8,5,19,16)-EPDF \\
    \hline
    81 & 2 & 40 & (81,40,19,20)-PDS & 20 & 4 & (81,10,4,3,0)-DPDF & (81,10,4,16,20)-EPDF \\
    \hline
    81 & 2 & 40 & (81,40,19,20)-PDS & 40 & 2 & (81,20,2,1,0)-DPDF & (81,10,4,18,20)-EPDF \\
    \hline
    81 & 4 & 20 & (81,20,1,6)-PDS & 40 & 2 & (81,10,2,1,0)-DPDF & (81,10,2,0,6)-EPDF \\
    \hline
    89 & 2 & 44 & (89,44,21,22)-PDS & 4 & 22 & (89,2,22,9,12)-DPDF & (89,2,22,12,10)-EPDF \\
    \hline
    89 & 2 & 44 & (89,44,21,22)-PDS & 8 & 11 & (89,4,11,2,8)-DPDF & (89,4,11,19,14)-EPDF \\
    \hline
    89 & 2 & 44 & (89,44,21,22)-PDS & 22 & 4 & (89,11,4,1,2)-DPDF & (89,11,4,20)-EDF \\
    \hline
    89 & 2 & 44 & (89,44,21,22)-PDS & 44 & 2 & (89,11,4,1,0)-DPDF & (89,11,20,22)-EPDF \\
    \hline
    97 & 2 & 48 & (97,48,23,24)-PDS & 4 & 24 & (97,2,24,9,14)-DPDF & (97,2,24,14,10)-EPDF \\
    \hline
    97 & 2 & 48 & (97,48,23,24)-PDS & 6 & 16 & (97,3,16,5,10)-DPDF & (97,3,16,18,14)-EPDF \\
    \hline
    97 & 2 & 48 & (97,48,23,24)-PDS & 8 & 12 & (97,4,12,3,8)-DPDF & (97,4,12,20,16)-EPDF \\
    \hline
    97 & 2 & 48 & (97,48,23,24)-PDS & 12 & 8 & (97,6,8,3,4)-DPDF & (97,6,8,20)-EDF \\
    \hline
    97 & 2 & 48 & (97,48,23,24)-PDS & 16 & 6 &  (97,8,6,3,2)-DPDF & (97,8,6,20,22)-EPDF \\
    \hline
    97 & 2 & 48 & (97,48,23,24)-PDS & 24 & 4 &  (97,12,4,1,2)-DPDF & (97,12,4,22)-EDF \\
    \hline
    97 & 2 & 48 & (97,48,23,24)-PDS & 32 & 3 &  (97,16,3,0,2)-DPDF & (97,16,3,23,22)-EPDF \\
    \hline
    97 & 2 & 48 & (97,48,23,24)-PDS & 48 & 2 &  (97,24,2,1,0)-DPDF & (97,24,2,22,24)-EPDF \\
    \hline
    101 & 2 & 50 & (101,50,24,25)-PDS & 4 & 25 &  (101,2,25,12)-DDF & (101,2,25,12,13)-EPDF \\
    \hline
    101 & 2 & 50 & (101,50,24,25)-PDS & 10 & 10 &  (101,5,10,4,5)-DPDF & (101,5,10,20)-EDF \\
    \hline
    101 & 2 & 50 & (101,50,24,25)-PDS & 20 & 5 &  (101,10,5,0,4)-DPDF & (101,10,5,24,21)-EPDF \\
    \hline
    101 & 2 & 50 & (101,50,24,25)-PDS & 50 & 2 & (101,25,2,0,1)-DPDF & (101,25,2,24)-EDF \\
    \hline
    109 & 2 & 54 & (109,54,26,27)-PDS & 4 & 27 & (109,2,27,12,14)-DPDF & (109,2,27,14,13)-EPDF \\
    \hline
    109 & 2 & 54 & (109,54,26,27)-PDS & 6 & 18 & (109,3,18,8,9)-DPDF & (109,3,18,18)-EDF \\
    \hline
    109 & 2 & 54 & (109,54,26,27)-PDS & 12 & 9 & (109,6,9,4)-DDF & (109,6,9,22,23)-EPDF \\
    \hline
    109 & 2 & 54 & (109,54,26,27)-PDS & 18 & 6 & (109,9,6,2,3)-DPDF & (109,6,9,24)-EDF \\
    \hline
    109 & 2 & 54 & (109,54,26,27)-PDS & 36 & 3 & (109,18,3,0,2)-DPDF & (109,18,3,26,25)-EPDF \\
    \hline
    109 & 2 & 54 & (109,54,26,27)-PDS & 54 & 2 & (109,27,2,0,1)-DPDF & (109,27,2,26)-EDF \\
    \hline
    113 & 2 & 56 & (113,56,27,28)-PDS & 4 & 28 & (113,2,28,15,12)-DPDF & (113,2,28,12,16)-EPDF \\
    \hline
    113 & 2 & 56 & (113,56,27,28)-PDS & 8 & 14 & (113,4,14,5,8)-DPDF & (113,2,28,22,20)-EPDF \\
    \hline
    113 & 2 & 56 & (113,56,27,28)-PDS & 14 & 8 & (113,7,8,3,4)-DPDF & (113,7,8,24)-EDF \\
    \hline
    113 & 2 & 56 & (113,56,27,28)-PDS & 16 & 7 & (113,8,7,2,4)-DPDF & (113,8,7,25,24)-EPDF \\
    \hline
    113 & 2 & 56 & (113,56,27,28)-PDS & 28 & 4 & (113,14,4,3,0)-DPDF & (113,14,4,24,28)-EPDF \\
    \hline
    113 & 2 & 56 & (113,56,27,28)-PDS & 56 & 2 & (113,28,2,1,0)-DPDF & (113,28,2,26,28)-EPDF \\
    \hline
    121 & 2 & 60 & (121,60,29,30)-PDS & 4 & 30 & (121,2,30,17,12)-DPDF & (121,2,30,12,18)-EPDF \\
    \hline
    121 & 2 & 60 & (121,60,29,30)-PDS & 6 & 20 & (121,3,20,13,6)-DPDF & (121,3,20,16,24)-EPDF \\
    \hline
    121 & 2 & 60 & (121,60,29,30)-PDS & 8 & 15 & (121,4,15,10,4)-DPDF & (121,4,15,19,26)-EPDF \\
    \hline
    121 & 2 & 60 & (121,60,29,30)-PDS & 10 & 12 & (121,5,12,5,6)-DPDF & (121,5,12,24)-EDF \\
    \hline
    121 & 2 & 60 & (121,60,29,30)-PDS & 12 & 10 & (121,6,10,9,0)-DPDF & (121,6,10,20,30)-EPDF \\
    \hline
    121 & 2 & 60 & (121,60,29,30)-PDS & 20 & 6 & (121,10,6,5,0)-DPDF & (121,10,6,24,30)-EPDF \\
    \hline
    121 & 2 & 60 & (121,60,29,30)-PDS & 24 & 5 & (121,12,5,4,0)-DPDF & (121,12,5,25,30)-EPDF \\    \hline
    121 & 2 & 60 & (121,60,29,30)-PDS & 30 & 4 & (121,15,4,1,2)-DPDF & (121,15,4,28)-EDF \\
    \hline
    121 & 2 & 60 & (121,60,29,30)-PDS & 40 & 3 & (121,20,3,2,0)-DPDF & (121,20,3,27,30)-EPDF \\
    \hline
    121 & 2 & 60 & (121,60,29,30)-PDS & 60 & 2 & (121,30,2,1,0)-DPDF & (121,30,2,28,30)-EPDF \\
    \hline
    121 & 3 & 40 & (121,40,15,12)-PDS & 6 & 20 & (121,2,20,11,4)-DPDF & (121,2,20,4,8)-EPDF \\
    \hline
    121 & 3 & 40 & (121,40,15,12)-PDS & 12 & 10 & (121,4,10,9,0)-DPDF & (121,4,10,6,12)-EPDF \\
    \hline
    121 & 3 & 40 & (121,40,15,12)-PDS & 15 & 8 & (121,5,8,3,2)-DPDF & (121,5,8,12,10)-EPDF \\
    \hline
    121 & 3 & 40 & (121,40,15,12)-PDS & 24 & 5 & (121,8,5,4,0)-DPDF & (121,8,5,11,12)-EPDF \\
    \hline
    121 & 3 & 40 & (121,40,15,12)-PDS & 30 & 4 & (121,10,4,3,0)-DPDF & (121,10,4,12)-EDF \\
    \hline
    121 & 3 & 40 & (121,40,15,12)-PDS & 60 & 2 & (121,20,2,1,0)-DPDF & (121,20,2,14,12)-EPDF \\
    \hline
    121 & 4 & 30 & (121,30,11,6)-PDS & 12 & 10 & (121,3,10,9,0)-DPDF & (121,3,10,2,6)-EPDF \\
    \hline
    121 & 4 & 30 & (121,30,11,6)-PDS & 24 & 5 & (121,3,10,4,0)-DPDF & (121,3,10,7,6)-EPDF \\
    \hline
    121 & 4 & 30 & (121,30,11,6)-PDS & 60 & 2 & (121,15,2,1,0)-DPDF & (121,15,2,10,6)-EPDF \\
    \hline
    121 & 6 & 20 & (121,20,9,2)-PDS & 12 & 10 & (121,2,10,9,0)-DPDF & (121,2,10,0,2)-EPDF \\
    \hline
    121 & 6 & 20 & (121,20,9,2)-PDS & 24 & 5 & (121,4,5,4,0)-DPDF & (121,4,5,5,2)-EPDF \\
    \hline
    121 & 6 & 20 & (121,20,9,2)-PDS & 60 & 2 & (121,10,2,1,0)-DPDF & (121,10,2,8,2)-EPDF \\
    \hline
\end{longtable}
}
\end{center}


\end{document}